\documentclass[a4paper,11pt,reqno]{amsart}

\usepackage{
amsfonts,
amsmath,
amsopn,
amssymb,
amsthm,
bbm,
bbold,
dsfont,
enumitem,
graphicx,
mathrsfs,
mathtools,
soul,
subfig,
verbatim,
xcolor,
xspace,
tikz,
mathabx
}

\usetikzlibrary{decorations.markings,decorations.pathmorphing,arrows.meta}

\usepackage{geometry}
\usepackage[hidelinks]{hyperref}
\usepackage[utf8]{inputenc}

\mathtoolsset{showonlyrefs}

\theoremstyle{plain} 
\newtheorem{thm}{Theorem}[section] 
\newtheorem{cor}[thm]{Corollary} 
\newtheorem{lem}[thm]{Lemma} 
\newtheorem{prop}[thm]{Proposition} 

\theoremstyle{definition}

\theoremstyle{definition}

\theoremstyle{remark} 
\newtheorem{rem}[thm]{Remark}

\newcommand{\f}[2]{\frac{#1}{#2}}
\newcommand{\tf}[2]{\tfrac{#1}{#2}}
\newcommand{\p}{\partial}
\newcommand{\bb}[1]{\bigg(#1\bigg)}

\newcommand{\qu}[1]{\left[#1\right]}
\newcommand{\gr}[1]{\left\{#1\right\}}
\newcommand{\scal}[3]{{\langle#1,#2\rangle}_{#3}}
\newcommand{\sn}[1]{\mathring{#1}}

\DeclareMathOperator*{\esup}{ess\,sup}

\newcommand{\N}{\mathbb{N}}
\newcommand{\R}{\mathbb{R}}
\newcommand{\C}{\mathbb{C}}
\newcommand{\D}{\mathcal{D}}
\newcommand{\T}{\mathcal{T}}
\newcommand{\LL}{\mathcal{L}}
\newcommand{\F}{\mathscr{F}}

\newcommand{\x}{\mathbf{x}}
\newcommand{\y}{\mathbf{y}}
\renewcommand{\k}{\mathbf{k}}
\newcommand{\z}{\mathbf{0}}
\newcommand{\ds}{\,ds}
\newcommand{\dt}{\,dt}
\newcommand{\dz}{\,dz}
\newcommand{\dteta}{\,d\theta}
\newcommand{\da}{\,d\alpha}
\newcommand{\dkm}{\,dk}
\newcommand{\dxm}{\,dx}
\newcommand{\dy}{\,dy}
\newcommand{\dtau}{\,d\tau}
\newcommand{\dome}{\,d\ome}

\newcommand{\dx}{\,{\rm d}\x}
\newcommand{\dsx}{\,{\rm dS}(\x)}
\newcommand{\dsy}{\,{\rm dS}(\y)}

\newcommand{\h}{\mathcal{H}}
\newcommand{\U}{\mathcal{U}}
\newcommand{\I}{\mathcal{I}}
\newcommand{\G}{\mathcal{G}^{\l}}
\newcommand{\Gn}{\mathcal{G}^{\nu}}
\newcommand{\g}{G^{\l}}
\newcommand{\gn}{G^{\nu}}
\newcommand{\somma}{\sum_{\ell=0}^\infty\sum_{m=-\ell}^\ell}

\newcommand{\ve}{\varepsilon}
\newcommand{\ep}{\varepsilon}
\newcommand{\al}{\alpha}

\newcommand{\lap}{\Delta}
\newcommand{\Ga}{\Gamma}
\newcommand{\de}{\delta}
\newcommand{\beq}{\begin{equation}}
\newcommand{\eeq}{\end{equation}}
\newcommand{\lf}{\left}
\newcommand{\ri}{\right}

\renewcommand{\b}[1]{\big(#1\big)}

\renewcommand{\i}{\imath}
\renewcommand{\exp}{{\rm e}}
\newcommand{\la}{\lambda}
\renewcommand{\l}{\lambda}
\renewcommand{\L}{\Lambda}
\newcommand{\ome}{\omega}

\renewcommand{\S}{\mathbb{S}^2}

\renewcommand{\leq}{\leqslant}
\renewcommand{\geq}{\geqslant}

\begin{document}

\title[Well--posedness of the 3D NLSE with sphere--concentrated nonlinearity]{Well--posedness of the three--dimensional NLS equation with sphere--concentrated nonlinearity}

\author[D. Finco]{Domenico Finco}
\address[D. Finco]{Universit\`a Telematica Internazionale Uninettuno, Facolt\`a di Ingegneria, corso Vittorio Emanuele II, 00186, Roma, Italy.}
\email{d.finco@uninettunouniversity.net}
\author[L. Tentarelli]{Lorenzo Tentarelli}
\address[L. Tentarelli]{Politecnico di Torino, Dipartimento di Scienze Matematiche ``G.L. Lagrange'', Corso Duca degli Abruzzi 24, 10129, Torino, Italy.} 
\email{lorenzo.tentarelli@polito.it}
\author[A. Teta]{Alessandro Teta}
\address[A. Teta]{Universit\`a degli Studi di Roma ``La Sapienza'', Dipartimento di Matematica ``G. Castelnuovo'', Piazzale Aldo Moro 5, 00185, Roma, Italy}
\email{teta@mat.uniroma1.it}

\date{\today}

\begin{abstract} 
 We discuss \emph{strong} local and global well--posedness for the three--dimensional NLS equation with nonlinearity concentrated on $\S$. Precisely, local well--posedness is proved for any $C^2$ power--nonlinearity, while global well--posedness is obtained either for small data or in the defocusing case under some growth assumptions. With respect to point--concentrated NLS models, widely studied in the literature, here the dimension of the support of the nonlinearity does not allow a direct extension of the known techniques and calls for new ideas.
\end{abstract}

\maketitle

\vspace{-.5cm}
\smallskip

    

\section{Introduction}
\label{sec-intro}

NLS equation is well known to provide an effective model for the description of microscopic systems on a macroscopic/mesoscopic scale, as for instance Bose--Einstein Condensates (e.g., \cite{DGPS-99}). However it is also employed with totally different physical meaning in nonlinear optics, plasma waves, neurosciences (FitzHugh--Nagumo model), etc (e.g., \cite{M-05} and references therein). In particular, several attempts have been made to adapt this model to the case of quantum many body systems in presence of defects or impurities with a spatial scale much smaller than the dispersion of the wave function.

Two different singular equations have been suggested in the last decades to address this problem. The former arises perturbing the laplacian in the NLS equation with a singular point potential of delta type (see \cite{SM-20,SCMS-20}). Available results concern here mainly the 1D case, where local/global well--posedness and existence/stability of standing waves is now well understood (e.g., \cite{ANR-20,AN-09,AN-13,ANV-13,LFFKS-08}). In 2D and 3D, on the contrary, first well--posedness results have been obtained in \cite{CFN-21}, whereas standing waves have been discussed in \cite{ABCT-21,ABCT-22,FGI-22}.

On the other hand, the latter arises formally multiplying the nonlinear term of the NLS equation by a point potential of delta type, thus causing a \emph{concentration} of the nonlinearity. This model has been first proposed to describe phenomena such as charge accumulation in semiconductor interfaces or heterostructures (e.g., \cite{BKB-96,JPS-95,MA-93,MB-02,N-98,PJC-91}), nonlinear propagation in defected Kerr--type media (e.g., \cite{SKB-99,SKBRC-01,Y-05}) and Bose--Einstein condensates in optical lattices with laser beams generated defects (e.g., \cite{DM-11,LKMF-11}). Local and global well--posedness, blow--up and existence/stability of the standing waves for this equation are nowadays almost completely understood in 1D (e.g., \cite{AFH-21,AT-01,CFT-19,HL-20,HL-21}) and 3D (e.g., \cite{ADFT-03,ADFT-04,ANO-13,ANO-16}), 
whereas a better understanding of the 2D case has been obtained only in the last years by \cite{ACCT-20,ACCT-21,CCT-19,CFT-17}. Note, finally, that also a non-autonomous variant of this model have been widely studied in the last twenty years (e.g., \cite{CCLR-01,CDFM-05,CCL-18,BCT-22})

In this paper we present the first discussion, to the best of our knowledge, on a generalized model where the actual physical dimension of the defect in the concentrated model is taken into account. Indeed, in real cases defects and impurities are more likely to be modeled by smooth and closed manifolds $\mathcal{M}$ embedded in $\R^3$, rather than zero--dimensional objects like points.

More in detail, such model can be formally expressed by the initial value problem
\begin{equation}
 \label{eq-formalone}
 \left\{
 \begin{array}{ll}
  \displaystyle\i\f{\p\psi}{\p t}=-\lap\psi+\beta|\psi|^{2\sigma}\psi\,\delta_{\mathcal{M}} & \text{in}\quad\R^+\times\R^3\\[.4cm]
  \displaystyle\psi(0,\cdot)=\psi_0 & \text{in}\quad \R^3,
 \end{array}
 \right.
\end{equation}
or, equivalently, by the nonlinear initial-boundary value problem
\begin{equation}
 \label{eq-formalsecond}
 \left\{
 \begin{array}{ll}
  \displaystyle\i\f{\p\psi}{\p t}=-\lap\psi & \text{in}\quad\R^+\times\R^3\setminus\{\mathcal{M}\}\\[.4cm]
  \displaystyle \psi^+=\psi^-=\psi & \text{in}\quad\R^+\times\mathcal{M}\\[.4cm]
  \displaystyle \f{\partial \psi^+}{\partial n}-\f{\partial \psi^-}{\partial n}=\beta|\psi|^{2\sigma}\psi & \text{in}\quad\R^+\times\mathcal{M}\\[.5cm]
  \displaystyle\psi(0,\cdot)=\psi_0 & \text{in}\quad \R^3,
 \end{array}
 \right.
\end{equation}
where $\psi^+$ is the restriction of $\psi$ to the region outside $\mathcal{M}$, while $\psi^-$ is the restriction of $\psi$ to the region inside $\mathcal{M}$.

However, even though this new model provides a more accurate description of physically--relevant cases, it involves several challenging mathematical obstacles. Indeed, the main advantage of classical point models, which is the complexity reduction to a zero--dimensional time--integral equation, here is completely lost. In this case, the model reduces to a time--space integral equation supported on the manifold where the nonlinearity is placed.

Such difference calls for new ideas in the proofs of both local and global well--posedness. For this reason, in this paper we limit ourselves to the case of the unit sphere, i.e. $\mathcal{M}=\S$. In this way, it is possible to develop a strategy relying on the spherical harmonics decomposition that allows to overcome the complexity generated by the physical dimension of the defect. On the other hand, although the simplification of the geometry of the manifold makes the problem more manageable, it still shares all the intrinsic issues connected to manifolds of codimension one, thus representing a suitable paradigm for future research.

\begin{rem}
 Besides representing a more accurate description of real world phoenomena, the study of \eqref{eq-formalone} (or \eqref{eq-formalsecond}) may be seen as the first step of a new justification for the point models mentioned at the beginning. Indeed, in place of considering concentrated nonlinearities as concentration limits of spread nonlinear potentials (e.g., \cite{CFNT-14,CFNT-17}), one could obtain them as singular limits of manifold--concentrated nonlinearities when the manifold shrinks to a point.
\end{rem}


\subsection{The linear case}
\label{subsec:linear}

As for the point models, in order to give a precise meaning to \eqref{eq-formalone} (or \eqref{eq-formalsecond}), one has to begin by rigorously defining the linear case, which has been studied in \cite{AGS-87} (for more general geometries see, e.g., \cite{BEHL-17,BEL-14,BLL-14}, while for different singular potentials see, e.g., \cite{HHS-99,SV-02}).

Formally, for any fixed $\alpha\in\R$ it reads

\begin{equation}
 \label{eq-cauchyf}
 \left\{
 \begin{array}{ll}
  \displaystyle\i\f{\p\psi}{\p t}=\b{-\lap+\al\,\delta_{\S}}\,\psi & \text{in}\quad\R^+\times\R^3\\[.4cm]
  \displaystyle\psi(0,\cdot)=\psi_0 & \text{in}\quad \R^3
 \end{array}
 \right.
\end{equation}
where $\delta_{\S}$ denotes the superficial Dirac delta distribution supported on the unit sphere of $\R^3$, denoted by $\S$, i.e.
\[
 \langle h\,\delta_{\S},\varphi\rangle:=\int_{\S}h(\y)\varphi(\y)\dsy,\qquad\forall\varphi\in C_0^\infty(\R^3).
\]
However, in order to state it in a more consistent way, according to quantum mechanics, it is necessary to define $-\lap+\al\,\delta_{\S}$ as suitable self--adjoint operator $\h_\al$ that extends in a nontrivial way $-\Delta_{|C_0^\infty(\R^3\setminus\S)}$ via the von Neumann--Krein theory of self--adjoint extensions.

To this aim, for any fixed $\la>0$, denote first by $\G$ the Green potential associated with the unit sphere of the operator $-\lap+\l$ in $\R^3$, i.e.
\begin{equation}
\label{eq-pot}
 \G h(\x):=\int_{\S}\g(\x-\y)\,  h(\y)\,\dsy,\qquad \forall\, h:\S\to\C,\quad\forall\x\in\R^3,
\end{equation}
where $\g$ is the Green function of $-\lap+\l$ in $\R^3$, i.e.
\begin{equation}
\label{eq-green}
 \g(\x):=\f{\exp^{-\sqrt{\l}\,|\x|}}{4\pi|\x|}.
\end{equation}
In view of this, $\h_\al:L^2(\R^3)\to L^2(\R^3)$ can be defined as the operator with domain
\begin{equation}
\label{domop}
\D (\h_\al) := \big\{ u\in L^2 (\R^3): \, \exists \la>0 \: \text{ s.t. }\: u+\al \G u_{|\S}=: \phi^\la\in H^2(\R^3) \big\}
\end{equation}
and action
\begin{equation}
\label{eq-actH}
\h_\al u =-\Delta\phi^\la+\al\la\G u_{|\S},\qquad\forall u\in \D (\h_\al).
\end{equation}
In other words, functions $u\in \D(\h_\al)$ can be decomposed in a \emph{regular part} $\phi^\la$, on which the action of the operator coincides with that of the standard Laplacian, and in a \emph{singular part} $-\al \G u_{|\S}$, on which the action of the operator is the multiplication times $-\lambda$.

We highlight that $\la$ is a dummy parameter as it does not actually affect the definition of $\left( \h_\al, \D (\h_\al) \right)$. To check this, consider $u\in \D(\h_\al)$. By definition, there exist $\lambda>0$ such that $u=\phi^\la-\al\G u_{|\S}$, with $\phi^\la\in H^2(\R^3)$. Fix, now, $\nu>0,\,\nu\neq\la$. By \eqref{eq-pot}
\begin{equation}
\label{eq-diffpot}
 \big(\Gn u_{|\S}-\G u_{|\S}\big)(\x)=\int_{\S}u_{|\S}(\y)\big(\gn(\x-\y)-\g(\x-\y)\big)\dsy,
\end{equation}
and thus, differentiating \eqref{eq-diffpot} and recalling that \eqref{eq-green} entails $\gn-\g\in H^2(\R^3)$, one can prove that $\Gn u_{|\S}-\G u_{|\S}\in H^2(\R^3)$ in turn. To this aim, it is sufficient to note that $u_{|\S}\in H^{3/2}(\S)$ (as pointed out by Remark \ref{rem-altroop}) and check that, if $f \in L^2(\R^3)$ and $g \in C^0(\S)$ then $h:= f\ast g\delta_{\S} \in L^2(\R^3)$. This is a consequence of Jensen's inequality and Fubini's theorem, since
\begin{align*}
\| h \|_{ L^2(\R^3)}^2
&= \int_{\R^3} \lf| \int_{\S}  f(\x-\y) g(\y)  \dsy \ri|^2 \dx \\
& \leq 4\pi  \int_{\R^3}  \int_{\S}  |f(\x-\y)|^2 \,  |g(\y)|^2  \dsy  \dx \\
& = 4\pi  \int_{\S} |g(\y)|^2  \int_{\R^3}   |f(\x-\y)|^2 \,   \dsy  \dx = 4\pi \|g \|_{L^2(\S)}^2 \|f \|_{L^2(\R^3)}^2.
\end{align*}
Moreover, if $f \in H^2(\R^3)$ and $g \in C^0(\S)$, then $h\in  H^2(\R^3)$ since $\Delta h= \Delta f\,\ast g\delta_{\S} $, and thus the claim follows setting $f=\gn-\g$ and $g=u_{|\S}$.

As a consequence, if one sets $\phi^\nu:=\phi^\la+\al\big(\Gn u_{|\S}-\G u_{|\S}\big)$, then $\phi^\nu\in H^2(\R^3)$. On the other hand, as $-\Delta\big(\Gn u_{|\S}-\G u_{|\S}\big)=\la\G u_{|\S}-\nu\Gn u_{|\S}$, one finds that $-\Delta\phi^\la+\al\la\G u_{|\S}=-\Delta\phi^\nu+\al\nu\G u_{|\S}$, so that one can conclude that the decompositions with $\la$ and $\nu$ are completely equivalent.

\begin{rem}
\label{rem-altroop}
 Another way to see the independence of $\lambda$ is given by the possibility to rewrite \eqref{domop}--\eqref{eq-actH} as follows:
\begin{multline}
\label{domop2}
 \D(\h_\alpha):=\bigg\{u\in L^2(\R^3): u^+\in H^2\big(\R^3\setminus \overline{B_1(\z)}\big),\,u^-\in H^2\big(B_1(\z)\big),\\
 u_{|\S}^+=u_{|\S}^-=u_{|\S},\,\f{\partial u^+}{\partial n}\bigg|_{\S}-\f{\partial u^-}{\partial n}\bigg|_{\S}=\alpha u_{|\S}\bigg\},
\end{multline}
\[
 \h_\alpha u(\x):=-\Delta u(\x),\qquad\forall\x\in\R^3\setminus\S,
\]
where $B_1(\z)$ is the unit ball centered at the origin and
\begin{equation}
\label{eq-upum}
\left\{
\begin{array}{ll}
 \displaystyle u^+:\R^3\setminus \overline{B_1(\z)}\to\C & \text{such that}\qquad u^+(\x)=u(\x),\quad\forall\x\in \R^3\setminus \overline{B_1(\z)},\\[.2cm]
 \displaystyle u^-:B_1(\z)\to\C & \text{such that}\qquad u^-(\x)=u(\x),\quad\forall\x\in B_1(\z).
\end{array}
\right.
\end{equation}
Such formulation is useful also for two further reasons. On the one hand, it shows, by standard Trace theory, that $u_{|\S}\in H^{3/2}(\S)$, which is useful in the previous computations on the independence of $\lambda$. On the other hand, it yields that $\nabla u=(\nabla u^+)\mathds{1}_{\R^3\setminus \overline{B_1(\z)}}+(\nabla u^-)\mathds{1}_{B_1(\z)}$, which entails that $\nabla u\in L^2(\R^3)$. However, throughout the paper we prefer the form with the decomposition for $\lambda > 0$, as it makes several computations easier.
\end{rem}

By Stone's theorem, self--adjointness of $\h_\al$ yields global well--posedness in $L^2 (\R^3)$ of
\begin{equation}
\label{eq-cauchylin}
 \left\{
 \begin{array}{ll}
  \displaystyle\i\f{\p\psi}{\p t}=\h_\al\psi & \text{in}\quad\R^+\times\R^3\\[.4cm]
  \displaystyle\psi(0,\cdot)=\psi_0 & \text{in}\quad \R^3
 \end{array}
 \right.
\end{equation}
for every $\psi_0\in L^2 (\R^3)$. In addition, it is usual to represent the solution of \eqref{eq-cauchylin} by means of the Duhamel formula, i.e.
\begin{equation}
 \label{eq-duhamel_l}
 \psi(t,\x)=U_t \psi_0 (\x)-\i \al\int_0^t\int_{\S}\U(t-s,\x-\y)\,{\psi_{|\S}(s,\y)}\dsy\ds,
\end{equation}
where $U_t$ denotes the free Schr\"odinger propagator of $\R^3$, i.e. the operator with integral kernel
\begin{equation}
 \label{eq-intk}
 \U(t,\x):=\f{\exp^{-\f{|\x|^2}{4\i t}}}{(4\pi\i t)^{\f{3}{2}}}.
\end{equation}
In particular, \eqref{eq-duhamel_l} clearly shows that the governing quantity of the problem is the function $q:=\psi_{|\S}:[0,+\infty)\times\S\to\C$, which is usually called \emph{charge}, and which allows to reconstruct the whole $\psi$ using \eqref{eq-duhamel_l} as a definition. In order to find the evolution equation for $q$ it is sufficient to trace \eqref{eq-duhamel_l} on $\S$, thus obtaining
\begin{equation}
 \label{eq-chargel}
 q(t,\x)=(U_t \psi_0)_{|\S}(\x)-\i \al \int_0^t\int_{\S}\U(t-s,\x-\y)\,{q(s,\y)}\dsy\ds,\qquad t\geq0,\:\x\in\S.
\end{equation}

Finally, we also mention that it is often convenient to rewrite \eqref{eq-chargel} in a more compact and operatorial way. To this aim, first, one introduces the function $\I:\R^+\times\S\times\S\to\C$ defined by
\begin{equation}
 \label{eq-kernI}
 \I(t,\x,\y):=\U(t;\x-\y)_{\big|\{|\x|=|\y|=1\}}=\f{\exp^{\f{\i}{2t}}\,\exp^{\f{\x\cdot\y}{2\i t}}}{(4\i\pi t)^{\f{3}{2}}},
\end{equation}
which allows to construct the family of operators $(I_t)_{t>0}$ that associates any integrable function $g:\S\to\C$ with the function $I_t g:\S\to\C$ such that
\begin{equation}
 \label{eq-I}
 I_t g(\x):=\int_{\S}\I(t,\x,\y)g(\y)\dsy,\qquad\forall\x\in\S.
\end{equation}
Then, one introduces the operator $\L$ that associates any function $g:[0,+\infty)\times\S\to\C$ with the function $\L(g):[0,+\infty)\times\S\to\C$ such that
\begin{equation}
 \label{eq-Lambda}
 \L g(t,\x):=\int_0^t \big(I_{t-s}g(s,\cdot)\big)(\x)\ds
\end{equation}
and set
\beq
\label{source}
F_0 (t,\x):= \lf(U_t \psi_0\ri)_{|\S}(\x),
\eeq
so that \eqref{eq-chargel} reads
\begin{equation}
 \label{eq-chargsl}
 q(t,\x)+\i\al \b{ \L q}(t,\x)= F_0 (t,\x).
\end{equation}


\subsection{Setting and main results}
\label{subsec:results}

Now, to define the nonlinear analogous of \eqref{eq-cauchylin} it is sufficient to set
\beq 
\al= \al(\psi):=\beta |  \psi|_{\S}|^{2\sigma}= \beta |q|^{2\sigma} \qquad \beta\in \R,\quad \sigma>0.
\eeq
As a consequence \eqref{eq-cauchylin} is replaced by
\begin{equation}
\label{eq-cauchy}
 \left\{
 \begin{array}{ll}
  \displaystyle\i\f{\p\psi}{\p t}=\h\psi & \text{in}\quad\R^+\times\R^3\\[.4cm]
  \displaystyle\psi(0,\cdot)=\psi_0 & \text{in}\quad \R^3,
 \end{array}
 \right.
\end{equation}
where $\h$ is no more a linear self--adjoint operator, but a nonlinear map (again independent of $\la$) with domain
\begin{align}
\label{eq-non_dom1}
\D ( \h) : & = \big\{ u\in L^2 (\R^3): \, \exists \la>0 \: \text{ s.t. }\: u+\G \nu(u_{|\S})=:\phi^\la \in H^2(\R^3)\big\} \\[.2cm]
           & = \big\{ u\in L^2 (\R^3): \, \exists \la>0,\,q:\S\to\C \: \text{ s.t. } \\[.2cm]
\label{eq-non_dom2}           &  \hspace{5cm}u+\G \nu(q)=:\phi^\la \in H^2(\R^3)\:\text{ and }\:q=u_{|\S}\big\}
\end{align}
and action
\begin{equation}
\label{eq-nonact}
\h u =-\Delta\phi^\la+\la\G\nu(q),\qquad\forall u=\phi^\la-\G \nu(q)\in \D(\h),
\end{equation}
where
\begin{equation}
\label{eq-nonlinearity}
\nu:\C\to\C,\qquad \nu(z): = \beta|z|^{2\sigma}z\quad\forall z\in\C.
\end{equation}
The Cauchy Problem in \eqref{eq-cauchy} thus represents the rigorous formulation of \eqref{eq-formalone} and \eqref{eq-formalsecond}.

\smallskip
In this paper we study \eqref{eq-cauchy} in $\D(\h)$, that is in a \emph{strong sense}, obtaining the following results.

\begin{thm}[Local Well--Posedness]
\label{main}
Let $\beta\in\R$, $\sigma\geq1/2$ and $\psi_0=\phi_0^\la-\G \nu(q_0) \in \D(\h)$. Then:
\begin{itemize}
\item[(i)] there exists $T>0$ such that \eqref{eq-cauchy} admits a unique strong solution 
\begin{equation}
\label{eq-regularity}
\psi\in C^0\big([0,T], \D(\h)\big)\cap C^1 \big([0,T], L^2(\R^3)\big);
\end{equation}
\item[(ii)] mass and energy are preserved along the flow, namely for all $t\in(0,T]$
\beq \label{massa}
M[\psi(t,\cdot)] :=\|\psi(t,\cdot) \|_{L^2(\R^3) }^2=M[\psi_0]
\eeq
\and 
\beq \label{energia}
E[\psi(t,\cdot)]: = \|\nabla\psi(t,\cdot)\|^2_{L^2(\R^3)}  +\f{\beta }{ \sigma+1} \|q(t,\cdot)\|_{L^{2\sigma+2}(\S)}^{2\sigma+2}=E[\psi_0].
\eeq
\end{itemize}
\end{thm}

\begin{rem}
\label{rem-opnorm}
 Note that $\D(\h)$ is not a vector space. Nevertheless, $C^0\big([0,T]; \D(\h)\big)$ is meant as the set of those functions $\psi:[0,T]\times\R^3\to\C$ such that $\psi(t,\cdot)\in\D(\h)$, for every $t\in[0,T]$, and such that $\|\psi(t+h,\cdot)-\psi(t,\cdot)\|_{\D(\h)}\to0$, as $h\to0$, for every $t\in[0,T]$, where $\|\cdot\|_{\D(\h)}:=\|\cdot\|_{L^2(\R^3)}+\|\h (\cdot)\|_{L^2(\R^3)}$.
\end{rem}

\begin{rem}
\label{rem-altramap}
 Note also that the energy is well defined for functions $u\in\D(\h)$. Indeed, the potential part is well defined by Sobolev embeddings in $\S$ (see Section \ref{sec-sobs}); whereas, arguing as in the linear case (see Remark \ref{rem-altroop}), one may rewrite \eqref{eq-non_dom1}--\eqref{eq-nonact} as
\begin{multline}
\label{dommap2}
 \D(\h):=\bigg\{u\in L^2(\R^3): u^+\in H^2\big(\R^3\setminus \overline{B_1(\z)}\big),\,u^-\in H^2\big(B_1(\z)\big),\\
 u_{|\S}^+=u_{|\S}^-=u_{|\S},\,\f{\partial u^+}{\partial n}\bigg|_{\S}-\f{\partial u^-}{\partial n}\bigg|_{\S}=\nu\big(u_{|\S}\big)\bigg\},
\end{multline}
\[
 \h u(\x):=-\Delta u(\x),\qquad\forall\x\in\R^3\setminus\S,
\]
where $B_1(\z)$ is the unit ball centered at the origin and $u^+,\,u^-$ are defined as in \eqref{eq-upum}, and thus $\nabla u\in L^2(\R^3)$. Furthermore, definition \eqref{dommap2} clearly shows the connection between \eqref{eq-cauchy} and \eqref{eq-formalsecond}.
\end{rem}

\begin{rem}
 Note finally that, combining \eqref{massa}, \eqref{energia} and Proposition \ref{carica}, one immediately finds that
 \[
  \psi\in C^0\big([0,T];H^1(\R^3)\big)
 \]
on any interval $[0,T]$ on which \eqref{eq-cauchy} is well--posed.
\end{rem}

\begin{thm}[Global Well--Posedness]
\label{main2}
Let $\beta\in\R$, $\sigma\geq1/2$ and $\psi_0=\phi_0^\la-\G \nu(q_0) \in \D(\h)$. Let also 
\begin{equation}
\label{eq-tmax}
 T^*:=\sup\big\{T>0:\text{item (i)$\&$(ii) of Theorem \ref{main} hold }\big\}.
\end{equation}
Then:
\begin{itemize}
\item[(i)] there exists $\ep>0$ such that, if $\| \phi_0^\la\|_{H^2(\R^3)}<\ep$, then $T^*=+\infty$;
\item[(ii)] if $\beta>0$ and $\sigma< 4/5$, then $T^*=+\infty$.
\end{itemize}
\end{thm}

\begin{rem}
\label{rem-bastauno}
 Note that the smallness assumption on the regular part of $\psi_0$, displayed by item (i) of Theorem \ref{main2}, tacitly implies a smallness assumption on the initial charge $q_0$ too. Indeed, by \eqref{eq-stima-contr} with $\eta={\phi_\la}_{|\S}$ and standard Trace inequalities, one can see that
 \[
  \|q_0\|_{H^{3/2}(\S)}\leq 2\|{\phi_\la}_{|\S}\|_{H^{3/2}(\S)}\leq C\|\phi_\la\|_{H^{2}(\R^3)}\leq C\ep
 \]
\end{rem}

The proofs of Theorems \ref{main} and \ref{main2} is based on a discussion of the features of the function $\psi$ defined by the nonlinear analogous of \eqref{eq-duhamel_l}, i.e.
\begin{equation}
 \label{eq-duhamel}
 \psi(t,\x)=U_t \psi_0 (\x)-\i\int_0^t\int_{\S}\U(t-s,\x-\y)\,\nu\b{q(s,\y)}\dsy\ds,
\end{equation}
whenever the charge $q$ is the solution of the nonlinear analogous of \eqref{eq-chargsl}, i.e.
\begin{equation}
 \label{eq-charge}
 q(t,\x)+\i\b{ \L \nu(q)}(t,\x)= F_0(t,\x),
\end{equation}
a.k.a. \emph{charge equation}, which thus arises as the governing equation of the model and is the center of our investigation.

Some further comments are in order. As for the point delta models, local well--posedness and conservation laws are proved both for the defocusing case, i.e. $\beta>0$, and for the focusing case, i.e. $\beta<0$. Unfortunately, in contrast to those model, here a lower bound on the power of the nonlinearity is required. This is due to the fact that one cannot apply the Fixed Point theorem to \eqref{eq-charge} with a sufficiently low spacial regularity to allow non $C^2$--nonlinearities, even using the well known metric--weakening trick by Kato (introduced by \cite{K-87,K-89}). More details are provided by Remark \ref{sigma}. We are not able to establish whether this is only a technical issue due to our use of the decomposition in spherical harmonics or not. Our guess is that the former guess is true, but its overcoming is out of reach at the moment.

On the other hand, also the results on the global well--posedness displays restrictions that are not present in the point delta models. Indeed, although for small initial data it is possible to prove it without further assumptions on $\beta$ and $\sigma$, for general data we have to limit ourselves to the defocusing case and, moreover, we have to require an upper bound on the power of the nonlinearity. The reason lies again in the technical issues connected to the physical dimension of the support of the nonlinearity, which makes more difficult the use of the classical blow--up alternative argument.

More precisely, in the defocusing case, unless one assumes $\sigma<4/5$, the energy conservation yields a--priori estimates on the time growth of the charge with respect to spatial regularities that are weaker with respect to the required one for the blow--up alternative in $\D(\h)$, which is $H^{3/2}(\S)$ (as we will see in Section \ref{sec-global}). At the moment it is not clear whether this threshold is optimal or not. Again, the guess coming from point delta models is that globality should hold for any power in the defocusing case. However, in contrast to point delta models, here the spatial regularity of the charge plays a crucial role and requires different strategies, which cannot be straightforwardly extended to all the powers (more details are provided by Remark \ref{rem-global}). On the other hand, even if unlikely, it is not possible to exclude in principle that the non dimensionless of the support of the delta give rise to phenomena of loss of regularity along the flow (which thus would prevent global well--posedness in $\D(\h)$). We plan to study in future papers the scattering and the possible existence of blow--up solutions in order to better understand the behavior of the problem for long times.

Finally, we also mention that Theorem \ref{main2} does not address the focusing case for general initial data. Here the missing tool is a suitable version of the Gagliardo--Nirenberg inequality that estimate the potential energy, which is concentrated on the sphere, by means of the mass and the kinetic energy, which are spread on the whole $\R^3$. Once more, this is an issue strongly related to the non dimensionless nature of the support of the nonlinearity. Also this problem will be addressed by our future research.

\begin{rem}
Note that we focus on power nonlinearities for the sake of simplicity and because they are the most relevant ones from a physical point of view. However, also more general types of nonlinearities could be considered with our strategy (as in \cite{K-95} for the standard NLS equation).
\end{rem}


\subsection{Organization of the paper}
\label{subsec:org}

The paper is organized as follows.
\begin{itemize}
 \item In Section \ref{sec-basis} we recall some classical topics of analysis on $\S$ in order to fix notations; precisely:
 \begin{itemize}
  \item Section \ref{sec-sobs} concerns spherical harmonics and Sobolev spaces on $\S$;
  \item Section \ref{subsec-function} concerns Sobolev spaces for functions of time and space, where the space variable varies on a sphere;
  \item Section \ref{sec-BF} concerns Bessel functions, the Bessel--Fourier transform and its connection with the Fourier transform.
 \end{itemize}
 \item In Section \ref{sec-prel} we discuss some preliminary tools that are necessary for the proofs of the main results; precisely:
 \begin{itemize}
  \item Section \ref{sec-opLambda} addresses the properties (of the operators $I_t$ defined by \eqref{eq-I} -- Lemmas \ref{lem-disp} and \ref{lem-sobI} -- and) of the operator $\L$ defined by \eqref{eq-Lambda} (Propositions \ref{prop-cont}, \ref{prop-reg1} and \ref{prop-reg2} and Corollary \ref{cor-reg2});
  \item Section \ref{subsec:qzero} addresses the regularity of the trace on $\S$ of functions in $\D(\h)$ (Proposition \ref{pro-BC});
  \item Section \ref{subsec-regsource} addresses the regularity of the source term of the charge equation defined by \eqref{source} (Proposition \ref{prop-source}).
 \end{itemize}
 \item In Section \ref{sec-local} we prove local well--posedness of \eqref{eq-cauchy}, that is item (i) of Theorem \ref{main}, thanks to a careful analysis of the properties of the charge equation \eqref{eq-charge} (Proposition \ref{carica}).
 \item In Section \ref{sec-cons} we show conservation of mass and energy along the flow, that is item (ii) of Theorem \ref{main}.
 \item In Section \ref{sec-global} we deal with global well--posedness of \eqref{eq-cauchy}, that is Theorem \ref{main2}; precisely,
 \begin{itemize}
  \item Section \ref{subsec-globalsmall} studies the case of small initial data (i.e., item (i));
  \item Section \ref{subsec-globaldef} studies the defocusing case for general initial data (i.e., item (ii)) via a blow--up alternative argument (Lemma \ref{lem-bualternative}).
 \end{itemize}
\end{itemize}
Furthermore, Appendix \ref{sec-schauder} presents the proof of \eqref{schauder_true}, which is a crucial tool throughout the paper, whereas Appendix \ref{app-formula} presents the proof of \eqref{eq-integralbessel}, which is necessary to prove the regularity transfer from the charge $q$ to the function $\psi$ defined by \eqref{eq-duhamel} (Proposition \ref{ricostruzione}).


\subsection*{Fundings and acknowledgements} L.T. has been partially supported by the INdAM GNAMPA project 2022 ``Modelli matematici con singolarit\`a per fenomeni di interazione'' (CUP E55F22000270001) and by the PRIN 2022 ``Nonlinear dispersive equations in presence of singularities (NoDES)'' (CUP E53D23005450006).
D.F. and A.T. have been partially supported by the PRIN 2022 ``Singular interactions and Effective Models in Mathematical Physics'' (CUP H53D23001980006) and also acknowledge the support of GNFM - INdAM. We also wish to thank William Borrelli and Fabio Nicola for helpful suggestions concerning Sobolev spaces on manifolds.

\subsection*{Data availability statement}
Data sharing not applicable to this article as no datasets were generated or analysed during the current study.
 
\subsection*{Conflict of interest}
On behalf of all authors, the corresponding author states that there is no conflict of interest.


\section{Basics of analysis on $\S$}
\label{sec-basis}

Before starting any discussion of the results stated in Section \ref{subsec:results}, it is worth fixing some notation and recalling some well known facts about the analytical tools usually involved in the study of problems on $\S$.
   

\subsection{Sobolev spaces on $\S$}
\label{sec-sobs}

A crucial role for the definition of Sobolev spaces on $\S$ is played by the so--called \emph{spherical harmonics} (see, e.g., \cite{NIST}). For every fixed $\ell\in\N$ and every fixed $m\in\{-\ell,\dots,0,\dots,\ell\}$, the spherical harmonic of order $\ell$ and $m$, which we denote by $Y_{\ell,m}$, is the function $Y_{\ell,m}:\S\to\C$, defined by
\[
 Y_{\ell,m}(\x):=(-1)^m\sqrt{\f{(2\ell+1)(l-m)!}{4\pi(\ell+m)!}}\,\exp^{\imath m\phi}\, P_{\ell,m}(\cos\theta)
\]
where $\theta\in[0,\pi]$ and $\phi\in[0,2\pi]$ are the \emph{colatitude} and the \emph{longitude} (respectively) associated with $\x$, and $P_{\ell,m}$ is the \emph{associated Lagrange polynomial} of order $\ell$ and $m$, namely the smooth solution of
\[
 (1-s^2)P_{\ell,m}''(s)-2sP_{\ell,m}'(s)+\qu{\ell(\ell+1)-\tf{m^2}{1-s^2}}P_{\ell,m}(s)=0.
\]
Such functions are known to be the eigenfunctions of the Laplace--Beltrami operator on $\S$, i.e.
\begin{equation}
 \label{eq-diagLB}
 -\Delta_{\S}Y_{\ell,m}=\ell(\ell+1)Y_{\ell,m},
\end{equation}
with
\[
 \Delta_{\S}:=\f{1}{\sin\theta}\f{\partial}{\partial\theta}\bigg(\sin\theta\f{\partial}{\partial\theta}\bigg)+\f{1}{\sin^2\theta}\f{\partial^2}{\partial\phi^2},
\]
and an orthonormal basis of $L^2(\S)$, so that for every $g\in L^2(\S)$
\begin{gather}
 g= \somma  g_{\ell,m}  Y_{\ell,m},\\[.2cm]
 \|g\|_{L^2(\S)}^2:=\int_{\S}|g(\x)|^2\dsx=\somma|g_{\ell,m}|^2,
\end{gather}
with
\[
 g_{\ell,m}:=\scal{g}{Y_{\ell,m}}{L^2(\S)}:=\int_{\S}g^*(\x)Y_{\ell,m}(\x)\dsx.
\]
Consequently (see, e.g., \cite[Section 1.7]{HLM-15}), for every $\mu\in\R\backslash\{0\}$ one can define the Sobolev spaces $H^\mu(\S)$ equivalently as
\begin{equation}
 \label{eq-sob}
 H^\mu(\S):=\bigg\{g\in L^2(\S):[g]_{\sn{H}^\mu(\S)}^2:=\somma\ell^{2\mu}|g_{\ell,m}|^2<\infty\bigg\},
\end{equation}
or
\begin{equation}
 \label{eq-sob1}
 H^\mu(\S):=\left\{g\in L^2(\S):[g]_{\sn{H}^\mu(\S)}^2:=\big\|(-\Delta_{\S})^{\mu/2}g\big\|_{L^2(\S)}^2<\infty\right\},
\end{equation}
where $(-\Delta_{\S})^{\mu/2}$ can be easily deduced by \eqref{eq-diagLB}, endowed with the natural norm $\|g\|_{H^\mu(\S)}^2:=\|g\|_{L^2(\S)}^2+[g]_{\sn{H}^\mu(\S)}^2$. Note that in the following we often use the further equivalent norm $\|g\|_{H^\mu(\S)}^2:=\somma\langle\ell\rangle^{2\mu}|g_{\ell,m}|^2$, where $\langle\cdot\rangle$ denotes the \emph{japanese brakets} (i.e., $\langle\ell\rangle:=\sqrt{1+\ell^2}$).

However, there is also another definition of $H^\mu(\S)$, which reads as follows (see again \cite[Section 1.7]{HLM-15}). Let $U_1,U_2$ be two open sets of $\S$ containing the northern and the southern emispheres of $\S$, respectively, and let $\varphi_{j}:U_{j}\to B$, $j=1,2$,  be two smooth diffeomorphisms, where $B$ denotes the unit ball of $\R^2$. Then,
\begin{equation}
 \label{eq-equiv_sob}
 H^\mu(\S)=\left\{g\in L^2(\S):\pi_{j}[\chi_{j}g]\in H^\mu(\R^2),\: j=1,2\right\},
\end{equation}
with
\begin{equation}
 \label{eq-normlab}
 \|g\|_{H^\mu(\S)}^2:=\|\pi_1[\chi_1g]\|_{H^\mu(\R^2)}^2+\|\pi_2[\chi_2g]\|_{H^\mu(\R^2)}^2,
\end{equation}
where
\[
 \pi_{j}[v](\x):=\left\{
 \begin{array}{ll}
  v\big(\varphi_{j}^{-1}(\x)\big), & \text{if }\x\in B,\\[.2cm]
  0,                  & \text{otherwise}
 \end{array}
 \right.,\qquad j=1,2,
\]
and $\{\chi_1,\chi_2\}$ is a partition of the unity associated with the two emispheres of $\S$ and such that ${\rm supp}\{\chi_{j}\}\subset U_{j}$, $j=1,2$.  This definition does not depend on $U_{j}$, $\varphi_{j}$ or $\chi_j$, in the sense that different choices yield equivalent norms, and is equivalent to \eqref{eq-sob} and \eqref{eq-sob1} (see, e.g., \cite{HLM-15,LM-72}). However, it has the advantage that, using partitions of unity and change of coordinates by the diffeomorphisms $\varphi_1\,\varphi_2$, one can easily extend the usual embedding theorems for Sobolev spaces from $\R^2$ to $\S$ (see, e.g., \cite{DPV-12,HLM-15}). Moreover, it also allows to prove directly some classical Schauder estimates. More precisely, when $\mu>1$, recalling \eqref{eq-nonlinearity} and using
\beq
\label{infi}
\|g\|_{L^\infty(\S)} \leq c \|g \|_{H^\mu(\S)},\qquad\forall g\in H^\mu(\S),
\eeq
one can check that
\beq \label{schauder_true}
\| \nu (g) \|_{ H^{\mu}(\S) } \leq c_\mu\,\| g \|^{2\sigma}_{ L^\infty(\S) } \| g \|_{ H^{\mu}(\S) },\qquad\forall g\in H^\mu(\S),
\eeq
whenever $\sigma\geq\f{[\mu]}{2}$, whence
\beq \label{schauder}
\| \nu (g) \|_{ H^{\mu}(\S) } \leq c_\mu\,  \| g \|_{ H^{\mu}(\S) }^{2\sigma+1}
\eeq
 (the proof is reported in Appendix \ref{sec-schauder} for the sake of completeness).

\begin{rem}
 In the following we will equivalently use $L^2(\S)$ and $H^0(\S)$ in order to denote the space of the square integrable functions on the unit sphere of $\R^3$, since this does not give rise to misunderstandings.
\end{rem}

Finally, we recall that it is also possible to prove that definitions \eqref{eq-sob1} and \eqref{eq-equiv_sob} can be also extended to general $W^{\mu,p}(\S)$ spaces with $p\in(1,+\infty)$, i.e.
\begin{align}
\label{eq-sobp}
 W^{\mu,p}(\S):= & \left\{g\in L^p(\S):[g]_{\sn{W}^{\mu,p}(\S)}^p:=\big\|(-\Delta_{\S})^{\mu/2}g\big\|_{L^p(\S)}^p<\infty\right\}\\[.2cm]
 \label{eq-sobp2} = & \left\{g\in L^p(\S):\pi_{j}[\chi_{j}g]\in W^{\mu,p}(\R^2),\: j=1,2\right\}
\end{align}
(see, e.g., \cite[Sections I.$5$ and I.$6$]{T-81} and \cite{S-67,S-83,A-98}).


\subsection{Function spaces on $I\times\S$}
\label{subsec-function}

As the main focus of the paper is the study of the time--dependent problem \eqref{eq-cauchy}, it is also convenient to recall Sobolev spaces for functions of time and space, where the space variable varies on the unit spheres; that is, Sobolev spaces for functions $g:I\times\S\to\C$, with $I$ an interval of the real line.

Exploiting definitions mentioned in the previous section, one can immediately define, for every $\mu\in\R$ and $\alpha>0$ fixed,
\[
 L^2\big(I,H^\mu(\S)\big):=\bigg\{g\in L^2(I\times\S):[g]_{L^2(I,\sn{H}^\mu(\S))}^2:=\int_{I}[g(t)]_{\sn{H}^\mu(\S)}^2\dt<+\infty\bigg\},
\]
and
\begin{multline*}
 H^\alpha\big(I,H^\mu(\S)\big):=\\[.4cm]
 =\left\{
 \begin{array}{ll}
  \displaystyle \left\{g\in L^2\big(I,H^\mu(\S)\big):g^{(j)}\in L^2\big(I,H^\mu(\S)\big),\,j\in\{1,\dots,\alpha\}\right\}, & \text{if }\alpha\in\N\setminus\{0\},\\[.6cm]
  \displaystyle \left\{g\in H^{[\alpha]}\big(I,H^\mu(\S)\big):\big[g^{([\al])}\big]_{\sn{H}^{\al-[\alpha]}\big(I,H^\mu(\S)\big)}<+\infty\right\}, & \text{if }\alpha\not\in\N,
 \end{array}
 \right.
\end{multline*}
where
\[
 [g]_{\sn{H}^{\al-[\alpha]}(I,H^\mu(\S))}^2:=\int_{I\times I}\f{\|g(t)-g(s)\|_{H^\mu(\S)}^2}{|t-s|^{1+2(\alpha-[\alpha])}}\ds\dt
\]
and $g(t)=g(t,\cdot):\S\to\C$ denotes the function that one obtains fixing the value of $t$. Note that these are Hilbert spaces when endowed with the natural norm.

It is also worth mentioning that the order in which the seminorms are considered can be exchanged. That is, for instance,
\begin{multline*}
 [g]_{\sn{H}^\al(I,\sn{H}^\mu(\S))}^2=\\[.4cm]
 =\left\{
 \begin{array}{ll}
  \displaystyle \int_{I\times I}\f{[g(t)-g(s)]_{\sn{H}^\mu(\S)}^2}{|t-s|^{1+2\alpha}}\ds\dt=\sum_{\ell=0}^\infty\sum_{m=-\ell}^\ell \ell^{2\mu}[g_{\ell,m}]_{\sn{H}^\alpha(\R)}^2, & \text{if }\alpha\in(0,1),\\[.6cm]
  \displaystyle \bigg[\f{\partial g}{\partial t}\bigg]_{L^2(I,\sn{H}^\mu(\S))}^2=\int_I\left[\tf{\partial g}{\partial t}(t)\right]_{\sn{H}^\mu(\S)}^2\dt=\sum_{\ell=0}^\infty\sum_{m=-\ell}^\ell \ell^{2\mu}\|g'_{\ell,m}\|_{L^2(I)}^2, & \text{if }\alpha=1,
 \end{array}
 \right.
\end{multline*}
and similarly when $\alpha>1$, where $g_{\ell,m}:\R\to\C$ are the functions defined by
\[
 g_{\ell,m}(t):=\scal{g(t)}{Y_{\ell,m}}{L^2(\S)}. 
\]
Moreover, whenever $I=\R$, one may write time Sobolev regularity by means of the Fourier transform with respect to the time $t$, i.e.
\[
\widehat{g} (\ome, \x) := \frac{1}{\sqrt{2\pi}} \int_\R \exp^{-\i \ome t} g(t,\x) \dt,
\]
and thus
\begin{equation}
\label{eq-timef}
[g]_{\sn{H}^\al(\R,\sn H^\mu(\S))}^2 = \int_\R |\ome|^{2\al} [\widehat{g} (\ome) ]_{\sn{H}^\mu(\S)}^2 \, d\ome= \somma \ell^{2\mu}  \int_\R |\ome|^{2\al} |\widehat{g}_{\ell,m}(\ome)|^2 \, d\ome.
\end{equation}

Finally, we also recall that
\begin{gather*}
 C^0\big(I,H^\mu(\S)\big):=\gr{g:I\times\S\to\C: \lim_{t\to s} \|g(t) -g(s) \|_{H^\mu (\S)} =0,\: \forall s \in I }\\[.2cm]
 L^\infty(I,H^\mu(\S)):=\gr{g:I\times\S\to\C:\esup_{t\in I}\|g(t)\|_{H^\mu(\S)}<+\infty},
\end{gather*}
endowed with the natural norms. The definition of $C^n\big(I,H^\mu(\S)\big)$ and $W^{n,\infty}(I,H^\mu(\S))$ with $n\in\N$ is straightforward.


\subsection{Bessel functions and Bessel--Fourier transform}
\label{sec-BF}

The last reminder concerns Bessel functions and the interaction between the spherical harmonics decomposition and the Fourier transform of $\R^3$, i.e.
\[
\F g(\k):=\f{1}{(2\pi)^{\f{3}{2}}}\int_{\R^3}\exp^{-\imath\x\cdot\k}g(\x)\dx.
\]

First, recall the definition of the Bessel function of the first kind of order $\eta$ (see, e.g., \cite{NIST}):
\[
 J_\eta(t):=\bb{\f{t}{2}}^\eta\sum_{j=1}^\infty\f{(-1)^j}{\Ga(j+1)\Ga(j+\eta+1)}\bb{\f{t}{2}}^{2j}.
\]
Note that, when $\eta$ is real and positive, Bessel functions satisfy the following estimates:
\begin{gather} 
  |J_\eta(x)|\leq\f{c}{|x|^{\f{1}{3}}},\qquad\forall \eta>0,\,x\in\R \label{landau1}\\[.2cm]
  |J_\eta(x)|\leq \f{c}{\eta^{\f{1}{3}}} ,\qquad\forall \eta>0,\,x\in\R \label{landau2}
 \end{gather}
with $c$ independent of $\eta$ and $x$ (see \cite{L-00} and \cite[pag. 357]{S-93}). Note that, allowing an $\eta$--dependence of $c$ in \eqref{landau1}, one could establish a stronger estimate in $|x|$, i.e.
\begin{equation}
\label{eq-unuseful}
|J_\eta(x)|\leq\f{c_\eta}{|x|^{\f{1}{2}}},\qquad\forall x>0,
\end{equation}
with $c_1(\eta)$ with a power--like growth at infinity (see, e.g., \cite{O-06}). However, as the uniformity of the constant with respect to $\eta$ is one on the main tools used throughout the paper, we will always prefer \eqref{landau1} to \eqref{eq-unuseful}.

We recall, now, that the Fourier transform preserves the orthogonal decomposition given by $(Y_{\ell,m})_{\ell,m}$. More precisely, using the Jacobi--Anger expansion of the plane wave in $\R^3$, i.e.
\begin{equation}
 \label{eq-plainw}
 \exp^{\i\,\x\cdot\y}=(2\pi)^{\f{3}{2}}\b{|\x||\y|}^{-\f{1}{2}}\sum_{\ell=0}^\infty\i^\ell J_{\ell+\f{1}{2}}\b{|\x||\y|}\sum_{m=-\ell}^\ell Y_{\ell,m}\left(\f{\x}{|\x|}\right)Y_{\ell,m}^*\left(\f{\y}{|\y|}\right)
\end{equation}
(see, e.g., \cite{F-17}), one can prove that, whenever $g(\x) =\somma f_{\ell,m} (|\x|) Y_{\ell,m} (\x / |\x|)$, there results 
\beq \label{fourier1}
\F g(\k) = \somma (-\i)^\ell \widetilde{g}_{\ell,m} (|\k|) Y_{\ell,m} \left(\f{\k}{|\k|}\right),
\eeq
where
\[
 \widetilde{g}_{\ell,m} (k):= \int_0^{+\infty} \frac{J_{\ell+1/2}(kr)}{\sqrt{kr}}r^2 \, g_{\ell,m}(r)\,dr
\]
is, up to the extra factor $r^2/\sqrt{r}$, Bessel--Fourier transform of $f$, a.k.a. the Hankel transform of $g$ (see also \cite[Chapter IV.3]{SW-71}). By analogous computations one can also see that
\[
\F^{-1}\F g  (\x) = \somma \i^\ell \widetilde{g}_{\ell,m} (|\x|) Y_{\ell,m} \left(\f{\x}{|\x|}\right).
\]
In addition, exploiting the orthogonality of the Bessel functions in $L^2\big([0,+\infty),r\,dr\big)$ (for an easy proof see \cite{P-14}), one can also show that $\|\widetilde{g}_{\ell,m}\|_{L^2 ([0,+\infty),r^2\,dr)}=\|g_{\ell,m}\|_{L^2 ([0,+\infty),r^2\,dr)}$, i.e. the Bessel--Fourier transform is unitary on $L^2 \big([0,+\infty),r^2\,dr\big)$, and that $\widetilde{\widetilde{g}}_{\ell,m}=g_{\ell,m}$, i.e. the Bessel--Fourier transform is involutory.

Finally, we note that \eqref{fourier1} implies:
\begin{itemize}
 \item[(i)] $(\F g )_{\ell,m}(|\k|)=(-\imath)^\ell\widetilde{g}_{\ell,m}(|\k|)$ and, for every borel function $\varphi:\R\to\R$,
 \beq \label{nota}
  (\F \varphi(-\Delta)\, g )_{\ell,m} (|\k|) = (-\imath)^\ell \varphi(|\k|^2) \widetilde{g}_{\ell,m}(|\k|),
 \eeq
 where $\varphi(-\Delta)$ is defined by standard Functional Calculus;
 \item[(ii)] whenever $h:\S \mapsto \C$ is sufficiently smooth, as it can be identified with the measure $h\delta_{\S}$, its Fourier transform $\F h :=\F h\delta_{\S}  : \R^3\to \C$ is well defined and smooth and there results
\begin{equation}
\label{eq-rinota}
 (\F h )_{\ell, m} (|\k|)=(-\imath)^\ell\, \widetilde{h}_{\ell,m}=(-\imath)^\ell \frac{h_{\ell,m}}{\sqrt{|\k|}}\, J_{\ell+1/2} (|\k|),
\end{equation}
which amounts to \cite[Eq. $(3.5.91)$]{S-15}; in particular $\F \delta_{\S} (\k) = \frac{ J_{1/2} (|\k|)}{\sqrt{|\k|}}$ (see \cite[Thm. 3.5.13 and following Remark]{S-15}).
\end{itemize}


\section{Preliminary results}
\label{sec-prel}

In this section we establish some technical results, which are required in the proofs of Theorems \ref{main} and \ref{main2}:
\begin{itemize}
 \item[(i)] the mapping properties of the operator $\Lambda$ present in the charge equation \eqref{eq-charge}, and defined by \eqref{eq-Lambda};
 \item[(ii)] the regularity of the trace on $\S$ of the functions in the domain of the nonlinear map $\h$, defined by \eqref{eq-non_dom1} (or \eqref{eq-non_dom2});
 \item[(iii)] the regularity of the source term of \eqref{eq-charge}, defined by \eqref{source}.
\end{itemize}


\subsection{Properties of the operator $\Lambda$}
\label{sec-opLambda}

The behavior of the solutions of \eqref{eq-charge} is strongly affected by the features of the operator $\Lambda$ defined by \eqref{eq-Lambda}. In particular, for our purposes the most relevant ones are the mapping properties between the function spaces defined in Section \ref{subsec-function}.

As a first step, we establish an $L^p(\S)$--$L^q(\S)$ estimate for the family of operators $(I_t)_{t>0}$ defined by \eqref{eq-I}. Preliminarily, we note that, using the Jacobi--Anger expansion of the plane wave of $\R^3$ given by \eqref{eq-plainw}, it is possible to check that $I_t$ acts as a multiplication operator with respect to the decomposition in spherical harmonics. Precisely,
\beq \label{symbol}
\lf( I_t g \ri)_{\ell,m} = \rho(t,\ell) \, g_{\ell,m}, \qquad\text{with}\qquad \rho(t,\ell)= \f{(-\i)^{\ell+3/2}\,\exp^{\frac{\i}{2t}}}{2t}J_{\ell +1/2} \lf( \frac{1}{2t} \ri) .
\eeq
This is not surprising since the integral kernel $\I(t,\x,\y)$ only depends on $\x\cdot\y$ and such kernels are well known to give rise to convolution operators, which thus can be diagonalized by a suitable transform. In the following we will often refer to $\rho(t,\ell)$ as the \emph{symbol} of the operator $I_t$. Note also that such a symbol could be also computed by the Funk--Hecke formula (see, e.g., \cite{DX-13}).

\begin{lem}
\label{lem-disp}
Let $t>0$, $p\in[1,2]$ and $r\in[2,+\infty]$, with $p^{-1}+r^{-1}=1$. Then, there exists $c>0$, independent of $p$ and $r$, such that
\beq \label{disp}
\| I_t g \|_{L^r (\S)} \leq \frac{c}{t^{\de(p)}} \, \|g\|_{L^p (\S)}, \qquad\forall g\in L^p(\S), 
\eeq
with
\begin{equation}
\label{eq-deltap}
 \de(p) := \frac{5}{3p} -\f16.
\end{equation}
\end{lem}

\begin{proof}
First, combining \eqref{symbol} and \eqref{landau1}, there results
\beq \label{stimaL2}
\| I_t g \|_{L^2 (\S)} \leq \sup_{\ell\in \N} |\rho(t, \ell)| \|g\|_{L^2 (\S)} \leq  \frac{c}{t^{2/3}} \|g\|_{L^2 (\S)}.
\eeq
On the other hand, by \eqref{eq-kernI} it is straightforward that
\beq \label{stimaL1}
\| I_t g \|_{L^\infty (\S)} \leq  \frac{c}{t^{3/2}} \|g\|_{L^1 (\S)}.
\eeq
Thus, by the Riesz--Thorin theorem (see, e.g., \cite[Theorem 2.1]{LP-15}), one obtains \eqref{disp}.
\end{proof}

In addition, we can also prove that the operators $I_t$ display a regularizing effect with respect to the Sobolev spacial regularity.

\begin{lem}
\label{lem-sobI}
Let $t>0$, $\mu\in\R$ and $z\in[0,1]$. Then, there exists $c>0$, independent of $\mu$ and $z$, such that
\beq
\label{stimasobo}
\| I_t g \|_{H^{\mu +\frac{1-z}{3} } (\S)}  \leq  \frac{c}{t^{1 -\f{z}{3}}} \|g\|_{H^{\mu} (\S)}, \qquad\forall g\in H^\mu(\S).
\eeq
\end{lem}

\begin{proof}
 First, interpolating \eqref{landau1} and \eqref{landau2}, there results
\[
| J_{\ell +1/2} (x)| \leq  \frac{c}{|x|^{\f{z}{3}} \langle \ell \rangle^{\f{1-z}{3}}  },
\]
with $c$ independent of $\mu$ and $z$. Hence
\beq \label{landau}
 | \rho(\ell, t)| \leq  \frac{c}{t^{1-\f{z}{3}} \langle \ell \rangle^{\f{1-z}{3}}  },
\eeq
which immediately implies \eqref{stimasobo}.
\end{proof}

Now, we discuss the properties of the operator $\L$ (defined by \eqref{eq-Lambda}) in the next three propositions.

\begin{prop}
 \label{prop-cont}
 Let $\mu\in\R$, $z\in(0,1]$, $r\in\big(\frac{3}{z},+\infty\big]$ and $T>0$. Then, there exists $C>0$, independent of $\mu$, $z$, $r$ and $T$, such that
\beq \label{completa}
 \|\L g \|_{L^\infty ([0,T],H^{\mu +\frac{1-z}{3} } (\S))} \leq C\bigg(\f{3(r-1)}{rz-3}\bigg)^{\f{r-1}{r}}T^{\frac{r-1}{r}-\f{3-z}{3} }\|g\|_{L^r ([0,T],H^\mu(\S))}
\eeq
for every $g\in L^r([0,T],H^\mu(\S))$. Moreover $\L g\in C^0 \big([0,T],H^{\mu +\frac{1-z}{3}} (\S)\big) $.
\end{prop}

\begin{proof}
Using \eqref{stimasobo}, one sees that for a.e. $t\in(0,T]$
\begin{equation}
\label{eq-stimautile}
\| \L g(t) \|_{H^{\mu +\frac{1-z}{3} } (\S)} \leq\int_0^t  \| I_{t-s} g(s) \|_{H^{\mu +\frac{1-z}{3} } (\S)} ds \leq c \int_0^t  \frac{1}{|t-s|^{1 -z/3}} \|g(s)\|_{H^{\mu} (\S)} ds.
\end{equation}
Thus, \eqref{completa} follows from H\"older inequality.

It is then left to prove that $\L g$ is continuous in $t$ with values in $H^{\mu +\frac{1-z}{3}}$. First, easy computations yield
\[
 \L g(t+h,\x)-\L g(t,\x)=\int_0^h\big(I_{t+h-s}g(s)\big)(\x)\ds+\int_0^t\big(I_{t-s}(g(s+h)-g(s))\big)(\x)\ds.
\]
Hence, arguing as before, one finds that
\begin{multline*}
 \|\L g(t+h)-\L g(t) \|_{H^{\mu +\frac{1-z}{3} } (\S)}\\[.2cm]
 \leq C\bigg(\f{3(r-1)}{rz-3}\bigg)^{\f{r-1}{r}}\bigg\{\left[(t+h)^{\frac{r-1}{r}-\f{3-z}{3}}-t^{\frac{r-1}{r}-\f{3-z}{3}}\right]\|g\|_{L^r ([0,h],H^\mu(\S))}+\\[.2cm]
 +t^{\frac{r-1}{r}-\f{3-z}{3}}\|g(\cdot+h)-g\|_{L^r ([0,t],H^\mu(\S))}\bigg\}.
\end{multline*}
Therefore, since when $h\to0$ the former term in the curly brakets converges to zero by the continuity of the powers and the absolute continuity of the Lebesgue integral and the latter converges to zero by the \emph{mean continuity property}, the claim is proved.
\end{proof}

\begin{rem}
\label{rem-esponenti}
 Note that by the assumptions on $z$ and $r$, $\f{3(r-1)}{rz-3}$ and, above all, $\frac{r-1}{r}-\f{3-z}{3} $ are positive. 
\end{rem}

\begin{rem}
 The mean continuity property is a well known property of $L^p$--spaces of functions with real values (see, e.g., \cite{O-71}). However, it can be easily generalized to \emph{Bochner spaces} using the density of smooth functions with respect to time (see, e.g., \cite[Section $5.2.9$]{E-10}).
\end{rem}

\begin{rem}
 As they play a crucial role throughout the paper, we single out the extremal cases of Proposition \ref{prop-cont}, i.e. $r=+\infty$ and $z=1$:
\beq
 \label{eq-inf_contr}
  \|\L g \|_{C^0 ([0,T],H^{\mu +\frac{1-z}{3} } (\S))}\leq \f{3C}{z}\, T^{\f{z}{3}}\|g\|_{L^\infty([0,T],H^\mu(\S))},
\eeq
and, for $r>3$,
\beq
 \|\L g \|_{C^0 ([0,T],H^{\mu  } (\S))}\leq C\bigg(\f{3r-3}{r-3}\bigg)^{\f{r-1}{r}} T^{\f{r-3}{3r}}\|g\|_{L^r ([0,T],H^\mu(\S))}. \label{alternativa}
\eeq
Clearly, combining the two results one obtains
\begin{equation}
 \label{eq-final}
 \|\L g \|_{C^0 ([0,T],H^{\mu} (\S))}\leq 3CT^{1/3}\|g\|_{L^\infty([0,T],H^\mu(\S))}.
\end{equation}
\end{rem}

\begin{prop}
 \label{prop-reg1}
 Let $\al,\,\mu\geq0$ and $T>0$. If $g \in H^\al\big(\R, H^\mu(\S)\big)$, with $\mathrm{supp}\{g\}\subseteq[0,T]\times\S$, then $\L g \in L_{loc}^2\big(\R, H^\mu(\S)\big)\cap \sn H^{\al+1/7}\big(\R, H^\mu(\S)\big)$.
\end{prop}
\begin{proof}
We can split the proof in two parts.

\emph{Part (i): $\L g \in L_{loc}^2\big(\R, H^\mu(\S)\big)$}. It is sufficient to fix arbitrary $T_1<T_2$ and prove that $\L g \in L^2\big([T_1,T_2], H^\mu(\S)\big)$. By \eqref{symbol}, we have
\begin{equation}
\label{eq-Lambdadec}
\lf(\L g (t) \ri)_{\ell, m} = \int_0^t \rho(t-s,\ell) g_{\ell,m}(s) \ds.
\end{equation}
As a first consequence, since $\mathrm{supp}\{g\}\subseteq[0,T]\times\S$, when $T_2\leq0$ the claim is straightforward. Fix, then, $T_2>0$. Using again that $\mathrm{supp}\{g\}\subseteq[0,T]\times\S$, on can write $\lf(\L g (t) \ri)_{\ell, m}$ as
\begin{equation}
\label{eq-lambdaA}
\lf(\L g (t) \ri)_{\ell, m} = 
\int_0^{T_2} H(t-s) \rho (t-s, \ell) g_{\ell,m}(s) \ds,\qquad\forall t\in[0,T_2]
\end{equation}
where $H(t)$ is Heavyside function. Now, observing that, by \eqref{landau1}, $|H(t-s) \rho (t-s, \ell)|\leq c|t-s|^{-2/3}$, with $c$ independent of $\ell$ and $m$, one can use the Schur test between $L^2(0,T_2)$ and $L^2(T_1,T_2)$ (see, e.g., \cite{HS-78}), with test functions identically equal to 1, to get that
\beq 
\| \lf(\L g (\cdot) \ri)_{\ell, m} \|_{L^2(T_1,T_2)} \leq  c_{T_1,T_2} \| g_{\ell,m}\|_{L^2(0,T_2)},
\eeq
with $c_{T_1,T_2}$ still independent of $\ell$ and $m$. As a consequence,
\begin{align*}
\| \L g \|_{ L^2([T_1,T_2], H^\mu(\S) ) }^2 &= \int_{T_1}^{T_2} \somma \langle \ell \rangle^{2\mu} \lf| \lf(\L g (t) \ri)_{\ell, m}\ri|^2 \dt  \\
&  = \somma  \langle \ell \rangle^{2\mu} \int_{T_1}^{T_2} \lf| \lf(\L g (t) \ri)_{\ell, m}\ri|^2 \dt \\
&  \leq c_{T_1,T_2}^2\somma \langle  \ell \rangle^{2\mu}\| g_{\ell,m}\|_{L^2(0,T_2)}^2 \\
& \leq c_{T_1,T_2}^2 \| g \|_{L^2([0,T_2], H^\mu (\S)}^2<+\infty.
\end{align*}

\emph{Part (ii): $\L g \in \sn H^{\al+1/7}\big(\R, H^\mu(\S)\big)$}. Combining \eqref{eq-lambdaA} with $T_2=T$ and the fact that $\mathrm{supp}\{g\}\subseteq[0,T]$, there results that
\[
\lf(\L g (t) \ri)_{\ell, m} = \big((H\rho(\cdot,\ell)) \ast g_{\ell,m} \big) (t),
\]
so that the Fourier transform with respect to time reads
\beq \label{fouprod}
\widehat{ \lf(\L g(\cdot)  \ri)_{\ell, m} }(\ome) = \widehat{H \rho(\cdot,\ell)}(\ome) \, \widehat{  g_{\ell,m} }(\ome).
\eeq
Let us discuss $\widehat{H \rho(\cdot,\ell)}$. By definition
\begin{equation}
\label{eq-trasfH}
 \widehat{H \rho(\cdot,\ell)}(\ome) = \frac{(-\i)^{\ell+3/2}}{\sqrt{8\pi}}\int_0^{+\infty} \exp^{-\i \ome t} \exp^{\i/2t} J_{\ell+1/2} \lf( \frac{1}{2t}\ri) \frac{\dt}{t}
\end{equation}
and thus
\[
 |\widehat{H \rho(\cdot,\ell)}(\ome)|\leq c\bigg(\underbrace{\int_0^{1}\bigg|J_{\ell+1/2} \lf( \frac{1}{2t}\ri)\bigg| \frac{\dt}{t}}_{=:I_1}+\underbrace{\int_1^{+\infty}\bigg|J_{\ell+1/2} \lf( \frac{1}{2t}\ri)\bigg| \frac{\dt}{t}}_{=:I_2}\bigg).
\]
Now, by \eqref{landau1}, it is straightforward that $I_1\leq c$, with $c$ independent of $\ell$. On the other hand, since from \cite[Eq. $10.9.4$]{NIST} there results that
\[
 \left|J_{\ell+1/2}\left( \frac{1}{2t}\right)\right|\leq\f{c}{4^{\ell+1/2}(\ell+1)!\,t^{\ell+1/2}},
\]
one can check that $I_2\leq c$, with $c$ independent of $\ell$, as well, and thus
\begin{equation}
 \label{eq-Hbound}
 |\widehat{H \rho(\cdot,\ell)}(\ome)|\leq c,\qquad\forall\ome\in\R.
\end{equation}
On the other hand, an easy change of variable yields
\[
 \widehat{H \rho(\cdot,\ell)}(\ome) = \frac{(-\i)^{\ell+3/2}}{\sqrt{8\pi}}\int_0^{+\infty} \exp^{\i \ome t} \exp^{-\i/2t} J_{\ell+1/2} \lf( \ome t \ri) \frac{\dt}{t},
\]
so that, for any fixed $a>0$, $\widehat{H \rho(\cdot,\ell)}(\ome)=\varphi_1(\omega)+\varphi_2(\omega)$, with
\[
\begin{array}{l}
\displaystyle \varphi_1(\ome) := \frac{(-\i)^{\ell+3/2}}{\sqrt{8\pi}}\int_0^a \exp^{\i \ome t} \exp^{-\i/2t} J_{\ell+1/2} \lf( \ome t \ri)\frac{\dt}{t},\\[.6cm]
\displaystyle \varphi_2(\ome): = \frac{(-\i)^{\ell+3/2}}{\sqrt{8\pi}}\int_a^{+\infty} \exp^{\i \ome t} \exp^{-\i/2t} J_{\ell+1/2} \lf( \ome t \ri)\frac{\dt}{t}.
\end{array}
\]
Using again \eqref{landau1}, one immediately obtains
\beq \label{g_2}
|\varphi_2(\ome) |\leq \f{c}{(a|\ome|)^{1/3}}.
\eeq
Moreover, integrating by parts, since $(-2\imath \exp^{-\imath/2t})'=\f{\exp^{-\imath/2t}}{t^2}$, we have 
\begin{align}
\int_0^a \exp^{\i \ome t} \exp^{-\i/2t} J_{\ell+1/2} \lf( \ome t \ri)\frac{\dt}{t} =& -2\i a \exp^{-\i/(2a) +\i \ome a} J_{\ell +1/2}(\ome a) \label{pezzo1} \\
& +2\i \int^a_0 \exp^{\i \ome t} \exp^{-\i/2t} J_{\ell+1/2} \lf( \ome t \ri)\dt \label{pezzo2} \\
& -2\ome \int^a_0\exp^{\i \ome t} \exp^{-\i/2t} t\, J_{\ell+1/2} \lf( \ome t \ri)\dt \label{pezzo3}\\
& +2\i\ome \int^a_0\exp^{\i \ome t} \exp^{-\i/2t} t\, J_{\ell+1/2}' \lf( \ome t \ri)\dt \label{pezzo4}.
\end{align}
and thus $\varphi_1=\varphi_1^1+\varphi_1^2+\varphi_1^3+\varphi_1^4$, with $\varphi_1^j$ given by the multiplication between $\frac{(-\i)^{\ell+3/2}}{\sqrt{8\pi}}$ and the r.h.s. of \eqref{pezzo1}, \eqref{pezzo2}, \eqref{pezzo3} and \eqref{pezzo4}, respectively. By \eqref{landau2}, one immediately sees that
\[
 |\varphi_1^1(\ome)|\leq ca,\quad |\varphi_1^2(\ome)|\leq ca\quad\text{and}\quad|\varphi_1^3(\ome)|\leq c|\ome| a^2,\qquad\forall \ome\in\R,
\]
while, in order to estimate $\varphi_1^4$ one first recalls that
\[
2  J_{\ell+1/2}'=  J_{\ell+3/2}-  J_{\ell-1/2}
\]
(see, e.g., \cite[Eq. $8.471.2$]{GR-07}) and then argues as before obtaining
\[
 |\varphi_1^4(\ome)|\leq c|\ome| a^2,\qquad\forall \ome\in\R.
\]
Therefore, setting $a=|\ome|^{-4/7}$, one finds that
\[
 |\widehat{H \rho(\cdot,\ell)}(\ome)|\leq c\bigg(\f{1}{|\ome|^{4/7}}+\f{1}{|\ome|^{1/7}}\bigg),\qquad\forall \ome\in\R,
\]
and thus, combining with \eqref{eq-Hbound}, there results
\begin{equation}
\label{eq-stimebuone}
  |\widehat{H \rho(\cdot,\ell)}(\ome)|\leq \f{c}{|\ome|^{1/7}},\qquad\forall \ome\in\R.
\end{equation}
Summing up,
\begin{align*}
[\L g]_{\sn H^{\al+1/7}(\R, H^\mu(\S)) }^2
&=  \somma \langle \ell \rangle^{2\mu} [(\L g(\cdot))_{\ell,m}]_{\sn  H^{\al+1/7}(\R) }^2 \\
&= \somma \langle \ell \rangle^{2\mu} \int_\R |\ome|^{2(\al+1/7)} \lf| \widehat{ (\L g(\cdot))_{\ell,m}}(\ome) \ri|^2\dome \\
&\leq c \somma \langle \ell \rangle^{2\mu} \int_\R |\ome|^{2\al} \lf| \widehat{  g_{\ell, m} }(\ome) \ri|^2 \dome
 = c\somma \langle \ell \rangle^{2\mu}  [ g_{\ell,m}]_{ H^{\al}(\R) }^2 \\
&= c\, [g]_{\sn H^{\al}(\R, H^\mu(\S)) }^2<+\infty,
\end{align*}
which completes the proof.
\end{proof}

\begin{prop}
 \label{prop-reg2}
 Let $\al,\,\mu\geq0$ and $T>0$. Let also $g \in H^\al\big([0,T], H^\mu(\S)\big)$. Then, whenever one of the following conditions is satisfied:
\begin{itemize}
\item[{(i)}] $0\leq\al<1/2$,
\item[{(ii)}] $1/2< \al <3/2$ and $g(0)=g(T)=0$,
\end{itemize}
there results that $\L g \in H^{\al+1/7}\big([0,T], H^\mu(\S)\big)$. Moreover, in both cases, there exists $c(T)>0$, only depending $T$, such that
\[
\|\L g\|_{ H^{\al+1/7}([0,T], H^\mu(\S)) }\leq c(T) \| g\|_{ H^{\al}([0,T], H^\mu(\S)) }.
\]
\end{prop}

\begin{rem}
 Note that the condition $g(0)=g(T)=0$ is well defined since
 \[
  H^\al\big([0,T], H^\mu(\S)\big)\hookrightarrow\C^0\big([0,T], H^\mu(\S)\big),\qquad\text{when}\quad\alpha>1/2,
 \]
 (where such embedding is established, e.g., in \cite[Section $1.7$]{CM-12} in the integer case and can be proved, e.g., by using \cite[Theorem 8.2]{DPV-12} on the components $g_{\ell,m}(\cdot)$ in the fractional case).
\end{rem}

\begin{proof}
First, we define the function $G:\R\times\S\to\C$ such that
\[
G(t,\x) :=
\begin{cases}
g(t,\x) & \text{if }(t,\x)\in [0,T]\times\S \\
0 & \text{otherwise}.
\end{cases}
\]
As a consequence,
\[
G_{\ell,m}(t) :=
\begin{cases}
g_{\ell,m}(t) & \text{if }t\in [0,T] \\
0 & \text{otherwise}.
\end{cases}
\]
Note also that, whenever $1/2< \al <3/2$ and $\|g(0)\|_{H^\mu(\S)}=\|g(T)\|_{H^\mu(\S)}=0$, $g_{\ell,m}(0)=g_{\ell,m}(T)=0$ for every $\ell$ and $m$, so that $G_{\ell,m}(0)=G_{\ell,m}(T)=0$ for every $\ell$ and $m$. Then, by \cite[Lemma 2.1]{CFNT-17}, we have that $G_{\ell, m} \in H^\al (\R)$ with $\| G_{\ell,m}\|_{H^\al (\R)} \leq c \| g_{\ell,m}\|_{H^\al (0,T)}$ ($c$ being independent of $\ell$ and $m$). Thus $G \in H^\al(\R, H^\mu(\S))$ and 
\beq \label{rogna}
\| G \|_{ H^\al(\R, H^\mu(\S)) }\leq c \|  g \|_{ H^\al([0,T], H^\mu(\S)) }.
\eeq
Now, since by construction $\mathrm{supp}\{G\}\subseteq[0,T]$, $G$ satisfies the assumptions of Proposition \ref{prop-reg2} and so $\L G \in H^{\al+1/7}([0,T], H^\mu(\S))$. Finally, since again by construction $\L G (t,\x) = \L g (t,\x)$ for all $(t,\x) \in [0,T]\times\S$, the proof is complete.
\end{proof}

Such a proposition has the following immediate corollay, which claims that one can actually drop the assumption $g(T)=0$ in the condition (ii).

\begin{cor}
 \label{cor-reg2}
 Let $\al,\,\mu\geq0$ and $T>0$. Let also $g \in H^\al\big([0,T], H^\mu(\S)\big)$. Then, whenever one of the following conditions is satisfied:
\begin{itemize}
\item[{(i)}] $0\leq\al<1/2$,
\item[{(ii)}] $1/2< \al <3/2$ and $g(0)=0$,
\end{itemize}
there results that $\L g \in H^{\al+1/7}\big([0,T], H^\mu(\S)\big)$. Moreover, in both cases, there exists $c(T)>0$, only depending $T$, such that
\[
\|\L g\|_{ H^{\al+1/7}([0,T], H^\mu(\S)) }\leq c(T) \| g\|_{ H^{\al}([0,T], H^\mu(\S)) }.
\]
\end{cor}

\begin{proof}
 Consider the function $G:\R\times\S\to\C$ defined by
\[
G(t,\x) :=
\begin{cases}
g(t,\x) & \text{if }(t,\x)\in [0,T]\times\S \\
g(2T-t,\x) & \text{if }(t,\x)\in [T,2T]\times\S \\
0 & \text{otherwise}.
\end{cases}
\]
so that,
\[
G_{\ell,m}(t) :=
\begin{cases}
g_{\ell,m}(t) & \text{if }t\in [0,T] \\
g_{\ell,m}(2T-t) & \text{if }t\in [T,2T] \\
0 & \text{otherwise}.
\end{cases}
\]
As a consequence, one can check that $\| G_{\ell,m}\|_{H^\al (0,2T)} \leq c \| g_{\ell,m}\|_{H^\al (0,T)}$, $c$ being independent of $\ell$ and $m$ (the proof is straightforward when $\alpha\in[0,1]\setminus\big\{\f{1}{2}\big\}$, whereas it requires to argue as in the proof of \cite[Lemma 2.1]{CFNT-17} to estimate the ``off--diagonal'' terms when $\alpha\in\big[1,\f{3}{2}\big)$). Thus $G \in H^\al([0,2T], H^\mu(\S))$, whence it satisfies the assumptions of Proposition \ref{prop-reg2} on the interval $[0,2T]$, which implies $\L G \in H^{\al+1/7}([0,2T], H^\mu(\S))$. Then, one concludes observing again that, by construction, $\L G (t,\x) = \L g (t,\x)$ for all $(t,\x) \in [0,T]\times\S$.
\end{proof}


\subsection{Regularity of the trace on $\S$ of functions in $\D(\h)$}
\label{subsec:qzero}

Here we discuss the regularity of $u_{|\S}$ for functions $u\in\D(\h)$, or equivalently, the regularity of $q$ in \eqref{eq-non_dom2}.

The natural guess, descending from the linear case, is that $q\in H^{3/2}(\S)$. To this aim, since the trace operator is bounded and surjective from $H^2(\R^3)$ to $H^{3/2}(\S)$, it is sufficient to prove that for any given function $\eta\in H^{3/2}(\S)$, there exists a unique $q\in H^{3/2}(\S)$ that solves, for some value of $\lambda>0$,
\[
 q(\x)+\G \nu(q ) (\x)=\eta(\x),\qquad\forall\x\in\S;
\]
that is,
\beq \label{BC}
q (\x) + \int_{\S}     \T^\l (\x,\y) \nu\big(q(\y)\big)  \dsy = \eta(\x),\qquad\forall\x\in\S,
\eeq
where $\T^\l (\x,\y):=G^\l (\x-\y)$ for every $\x, \, \y\in \S$. Therefore, we state the following result.

\begin{prop} \label{pro-BC}
Let $\beta\in\R$, $\sigma\geq1/2$ and $\eta\in H^{3/2}(\S)$. Then, there exists $\l>0$ such that there exists a unique $q\in H^{3/2} (\S)$ that solves \eqref{BC}. In addition
\begin{equation}
 \label{eq-stima-contr}
 \|q\|_{H^{3/2} (\S)}\leq 2\|\eta\|_{H^{3/2} (\S)}.
\end{equation}
\end{prop}

\begin{proof}
Preliminarily, we denote by $T^\l$ the operator defined by the integral kernel $\T^\l (\x,\y)$. Then, we may divide the proof in three steps.

\emph{Step (i): $L^\infty(\S)$--estimate for $T^\l$}. We aim at proving
\beq \label{stimuccia}
\| T^\l g \|_{  L^\infty (\S) } \leq \f{c}{\sqrt{\l}}\|  g \|_{  L^\infty (\S) },\qquad\forall g\in L^\infty (\S).
\eeq
Since by \eqref{eq-green} $ \T^\l (\x,\y)\geq 0$, it is sufficient to prove that 
\[
\int_{\S} \T^\l (\x,\y) \dsy \leq \f{c}{\sqrt{\l}},\qquad \text{for a.e.}\quad \x \in \S.
\]
For any fixed $\x\in\S$, choosing the angle between $\x$ and $\y$ as the colatitude $\theta$, there results
\begin{align*}
\int_{\S} \T^\l (\x,\y) \dsy & = \f{1}{2\sqrt2} \int_0^\pi \f{ \exp^{-\sqrt{2 \l (1-\cos \theta)}}}{ \sqrt{1-\cos \theta}}\sin\theta\,d\theta 
=  \f{1}{2\sqrt2} \int^2_0 \f{ \exp^{-\sqrt{2\l a}}}{\sqrt{a}}\,da \\
&\leq  \f{1}{2\sqrt2} \int^{+\infty}_0 \f{ \exp^{-\sqrt{2\l a}}}{\sqrt{a}}\,da =  \f{\sqrt{2}}{4\sqrt{\l}},
\end{align*}
which, thus, proves \eqref{stimuccia}.

\emph{Step (ii): $H^\mu(\S)$--estimates for the operator defined by $\T^\l$}. We aim at proving
\beq \label{stimalineare} 
\| T^\l g\|_{H^\mu (\S) } \leq c \, \l^{-1/3} \|g \|_{H^\mu (\S)},\qquad\forall g\in H^\mu (\S),
\eeq
for any fixed $\mu\in(1,2)$. As we made for $I_t$ in Section \ref{sec-opLambda}, we start by establishing a suitable representation of the operator with respect to the decomposition in spherical harmonics. First, by \eqref{nota} we have that
\[
 \F (-\Delta+\lambda)^{-1}u (\k)=\sum_{\ell=0}^\infty\sum_{m=-\ell}^\ell (-\imath)^\ell\f{\widetilde{u}_{\ell,m}}{|\k|^2+\lambda}Y_{\ell,m}\bigg(\f{\k}{|\k|}\bigg).
\]
Furthermore, using \eqref{eq-plainw} and the $L^2(\S)$--orthonormality of the spherical harmonics, there results
\begin{align*}
 (-\Delta+\lambda)^{-1}u\,(\x) & =\F^{-1} \F (-\Delta+\lambda)^{-1}u (\x)\\[.2cm]
 & =\sum_{\ell=0}^\infty\sum_{m=-\ell}^\ell Y_{\ell,m}\bigg(\f{\x}{|\x|}\bigg)\int_0^{+\infty} \frac{k^2}{\sqrt{k|\x|}} J_{\ell+1/2}(k|\x|) \frac{\widetilde{u}_{\ell,m}(k)}{k^2+\l} \dkm
\end{align*}
and thus
\beq \label{fourier2}
\big((-\Delta+\lambda)^{-1}u\big)_{\ell,m} (|\x|) =\int_0^{+\infty} \frac{k^2}{\sqrt{k|\x|}} J_{\ell+1/2}(k|\x|) \frac{\widetilde{u}_{\ell,m}(k)}{k^2+\l} \dkm.
\eeq
Now, since by standard potential theory and \eqref{eq-pot}
\[
 (-\Delta+\lambda)\G g=g\delta_{\S}
\]
and since (arguing as at the end of Section \ref{sec-BF}) \eqref{fourier2} can be proved also for $u=g\delta_{\S}$ with $g\in H^\mu(\S)$, using \eqref{eq-rinota} there results
\[
 (\G g)_{\ell,m}(|\x|)=\f{g_{\ell,m}}{\sqrt{|\x|}}\int_0^{+\infty}\f{k J_{\ell+1/2}(k)J_{\ell+1/2}(k|\x|)}{k^2+\lambda}\dkm.
\]
Finally, since $T^\lambda g\,(\x)=\G g_{|\S}\,(\x)$, for every $\x\in\S$, we have that
\beq \label{actionT}
\big( T^\l g\big)_{\ell,m} = T^\l_\ell \, g_{\ell, m}, \qquad\text{with}\qquad T^\l_{\ell}:=\int_0^{+\infty} \frac{k J^2_{\ell+1/2} (k)}{k^2+\l}\dkm.
\eeq
As a consequence, using \eqref{landau1} we obtain
\beq \label{stimetta}
|\T^\l_{\ell}|\leq c \int_0^{+\infty} \frac{k^{1/3} }{(k^2+\l) } \, {\rm d}k \leq c \l^{-1/3},
\eeq
with $c$ independent of $\ell$, which immediately yields \eqref{stimalineare}.

\emph{Step (iii): claim of the proposition}. In order to complete the proof it is sufficient to show that there exists $\lambda>0$ such that the map
\[
{\tau}^\l( q) := -T^\l \big(\nu( q)\big)+ \eta
\]
is a contraction in 
\[
X:=\left\{ q\in H^{3/2}(\S):\|q\|_{H^{3/2}(\S)} \leq 2 \| \eta \|_{H^{3/2}(\S)}\right\}, 
\]
with respect to a proper metric which make $X$ complete. Let us prove, first, that $X$ is preserved by $\tau^\lambda$. By \eqref{stimalineare} and \eqref{schauder}, we have that
\begin{equation}
\label{eq-preservation}
\|  \tau^\l( q) \|_{H^{3/2}(\S)}  \leq c_{\beta,\sigma}\, \big(2 \| \eta \|_{H^{3/2}(\S)}\big)^{2\sigma+1} \l^{-1/3} + \| \eta^\l \|_{H^{3/2}},\qquad\forall q\in X,
\end{equation}
and, thus, the claim is proved for a sufficiently large $\l$. It is left to discuss contractivity with respect to a suitable metric. Consider the $L^\infty(\S)$--one. It is not difficult to see that it makes $X$ complete. Indeed, an $L^\infty(\S)$--Cauchy sequence in $X$ converges in $L^\infty(\S)$ to a limit which has to belong to $X$ since the sequence is also weakly convergent in $H^\mu(\S)$ by Banach--Alaoglu. Then, fix arbitrary $q_1,\,q_2\in \big(X,\|\cdot\|_{L^\infty(\S)}\big)$. Since
\[
 \tau^\lambda(q_1)-\tau^\lambda(q_2)=T^\l \big(\nu( q_2)-\nu( q_1)\big),
\]
combining \eqref{stimalineare}, \eqref{infi} and the fact that
\begin{equation}
\label{eq-nonlin}
 \lf| |z_1|^{2\sigma}z_1 -  |z_2|^{2\sigma}z_2 \ri| \leq c\big(|z_1|^{2\sigma} + |z_2|^{2\sigma} \big)|z_1 -z_2|,\qquad\forall z_1,\,z_2\in\C,
\end{equation}
one can check that
\[
 \|\tau^\lambda(q_1)-\tau^\lambda(q_2)\|_{L^\infty(\S)}\leq c\l^{-1/3}\| \nu(q_1)-\nu(q_2) \|_{L^\infty ( \S )}\leq c_{\beta,\sigma}\l^{-1/3}\| q_1 - q_2 \|_{L^\infty ( \S )} .
\]
Hence, $\tau^\l$ is contractive again for $\l$ large enough.
\end{proof}

\begin{rem}
\label{rem-invla}
 Note that, since $\nu(q)\in H^{3/2}(\S)$ whenever $q\in H^{3/2}(\S)$ (from \eqref{schauder_true}), one can argue as in the linear case to prove that $\D(\h)$ is independent of $\lambda$.
\end{rem}

\begin{rem}
Note also that, combining Proposition \ref{pro-BC} with the fact that the $H^{3/2}(\S)$--regularity is preserved by $\nu$ (again by \eqref{schauder_true}), the knowledge of the regular part of $u\in \D (\h)$ for some fixed $\lambda$ allows to reconstruct $q$, and thus $u$ itself. Note that the converse is false in general (as a consequence of Remark \ref{rem-invla}).
\end{rem}


\subsection{Regularity of the source term of (\ref{eq-charge})}
\label{subsec-regsource}

Here we discuss the regularity of the source term $F_0$ of \eqref{eq-charge}, defined by \eqref{source}. Clearly, since $\psi_0\in\D(\h)$, for any fixed $\l>0$ we can decompose it as $F_0=F_{0,1}+F_{0,2}$, where
\[
\left\{
\begin{array}{l}
 \displaystyle F_{0,1} (t,\x):= (U_t\phi_{0}^\l  )_{|\S}(\x),\\[.2cm]
 \displaystyle F_{0,2} (t,\x) :=-(U_t\G \nu(q_0))_{|\S}(\x),
\end{array}
\right.
\qquad\forall (t,\x)\in[0,+\infty)\times\S,
\]
with $\phi_{0}^\l\in H^2(\R^3)$ and (by Proposition \ref{pro-BC}) $q_0\in H^{3/2}(\S)$.

As in the previous results, it is convenient to decompose these quantities with respect to the basis of the spherical harmonics. Using \eqref{eq-plainw}, \eqref{nota}, \eqref{eq-rinota} and the $L^2(\S)$--orthonormality of the spherical harmonics as in Step (ii) of the proof of Proposition \ref{pro-BC}, and recalling that by functional calculus $U_t=\exp^{-\i\Delta t}$ and $U_t\G \nu(q_0)=\big(\exp^{-\i\Delta t}(-\Delta+\lambda)^{-1}\big)\nu(q_0)\delta_{\S}$, one can check that
\begin{align}
 \label{eq-F1}(F_{0,1})_{\ell,m}(t)&= \int_0^{+\infty} r^{3/2}\exp^{-\i t r^2} J_{\ell+1/2} (r) (\widetilde{\phi_0^\l})_{\ell,m} (r) \,dr\\[.2cm]
 (F_{0,2})_{\ell,m}(t)&=f_{2,\ell}(t)\big(\nu(q_0)\big)_{\ell,m},\quad\text{with}\quad f_{2,\ell}(t):=\int_0^{+\infty}  \frac{r\,\exp^{-\i t r^2}}{r^2+\l}  J^2_{\ell+1/2} (r)\,dr.
\end{align}
In addition, note that, since $q_0:={\psi_0}_{|\S}$ and $U_0=\mathbb{I}$, one has $F_0(0)\equiv q_0$.

\begin{prop} \label{prop-source}
Let $\beta\in\R$, $\sigma\geq1/2$ and $\psi_0=\phi_{0}^\l-\G \nu(q_0) \in \D(\h)$. Then, for any fixed $\l>0$,
\begin{align}
& \| F_0 \|_{C^0([0,T], H^{3/2}(\S)) }\lesssim_{\l} \big(\| \varphi_0^\l \|_{H^{2}(\R^3)}+   \|q_0\|_{H^{3/2}(\S)}^{2\sigma+1} \big) \label{sourceinfinito},\\[.2cm]
& \|F_0 \|_{ H^\al ([0,T],H^{2/3}(\S) )} \lesssim_{\l,\alpha,T} \big(\| \varphi_0^\l \|_{H^{2}(\R^3)}+   \|q_0\|_{H^{3/2}(\S)}^{2\sigma+1} \big), \label{fase1}\\[.2cm]
& \|F_0-q_0 -\i \L \nu(q_0) \|_{ H^1 ([0,T],H^{2/3}(\S) )} \lesssim_{\l,T} \big(\| \varphi_0^\l \|_{H^{2}(\R^3)}+   \|q_0\|_{H^{3/2}(\S)}^{2\sigma+1} \big)\label{fase2},
\end{align}
for every $T>0$ and every $\alpha\in\big(0,\tfrac{9}{14}\big)$.
\end{prop}

\begin{rem}
 Although $\L$ is defined in \eqref{eq-Lambda} for functions of time and space, the extension to functions of space only is straightforward. In particular
 \[
  \big(\L \nu(q_0)\big)(t,\x):=\int_0^t\big(I_{t-s}\nu(q_0)\big)(\x)\ds.
 \]
\end{rem}

\begin{proof}
We divide the proof in three parts according to the three claims of the statement.

\emph{Part (i): proof of \eqref{sourceinfinito}}. On the one hand, since $\phi_{0}^\l\in H^2(\R^3)$, by the Stone's theorem $U_t \phi_0^\l \in C^0(\R,H^2(\R^3))$ with $\|U_{(\cdot)} \varphi_0^\l\|_{C^0([0,T], H^{2}(\S))}\lesssim \| \varphi_0^\l \|_{H^{2}(\R^3)}$, for every $T>0$. Hence, by standard Trace theory, $F_{0,1}\in C^0\big(\R, H^{3/2}(\S)\big) $ with $\| F_{0,1} \|_{C^0([0,T], H^{3/2}(\S)) }\lesssim\| \varphi_0^\l \|_{H^{2}(\R^3)} $, for every $T>0$. On the other hand, arguing as in \eqref{stimetta}, one can see that $f_{2,\ell}$ are continuous functions such that
\[
| f_{2,\ell}(t) | \leq c_\l,\qquad\forall t\in\R,
\]
with $c_\l$ independent from $\ell$. As a consequence, for every fixed $t\in\R$, $\|F_{0,2}(t+h)-F_{0,2}(t)\|_{H^{3/2}(\S)}\to0$, as $h\to 0$, by dominated convergence and, furthermore,
\[
 \|F_{0,2}(t)\|_{H^{3/2}(\S)}\leq c_\l\|\nu(q_0)\|_{H^{3/2}(\S)}\leq c_\l\|q_0\|_{H^{3/2}(\S)}^{2\sigma+1},\qquad\forall t\in\R,
\]
thus proving \eqref{sourceinfinito} (one could actually prove that $F_{0,2}\in C^\delta\big(\R, H^{3/2}(\S)\big)$, for all $\delta\in \big(0,\f{1}{3}\big)$, but this goes beyond the aims of the proof).

\emph{Part (ii): proof of \eqref{fase2}}. Preliminarily, we note that, since $F_0(0)\equiv q_0$, we can rewrite $F_0$ as
\[
F_0(t,\x)-q_0(\x) -\i \big(\L \nu(q_0)\big)(t,\x)= G_1(t,\x) + G_2(t,\x),\qquad\forall (t,\x)\in\R\times\S,
\]
with
\[
G_1(t):=F_{0,1} (t,\x)- F_{0,1} (0,\x) \quad\text{and} \quad G_2 (t,\x): = F_{0,2} (t,x)-F_{0,2} (0,x) -\i \big(\L \nu(q_0)\big)(t,\x).
\]
Again, the main tool is to establish the decomposition of these quantities with respect to the basis of the spherical harmonics. A straightforward computation shows that
\begin{equation}
\label{eq-G1}
 (G_1)_{\ell,m}(t)= \int_0^{+\infty} r^{3/2}\big(\exp^{-\i t r^2}-1\big) J_{\ell+1/2} (r) (\widetilde{\phi_0^\l})_{\ell,m} (r) \,dr,
\end{equation}
while for $(G_2)_{\ell,m}(t)$ some further effort is required. First we see that
\[
 (F_{0,2})_{\ell,m}(t)-(F_{0,2})_{\ell,m}(0)=\big(\nu(q_0)\big)_{\ell,m}\int_0^{+\infty}  \frac{r\,\big(\exp^{-\i t r^2}-1\big) }{r^2+\l}  J^2_{\ell+1/2} (r)\,dr.
\]
In addition, from the Spectral theorem we have that
\[
-\i \int_0^t U_{t-s} \ds= \big(U_t-\mathbb{I}\big)(-\Delta+\ep)^{-1}-\imath\ep\int_0^tU_{t-s}(-\Delta+\ep)^{-1}\ds,
\]
for any fixed $\ep>0$, while, from \eqref{eq-Lambda}, \eqref{eq-I} and \eqref{eq-kernI}, we have that
\[
 -\i \big(\L \nu(q_0)\big)(t,\cdot)=\bigg(-\imath\int_0^t\big(U_{t-s}(\nu(q_0)\delta_{\S})\big)(\x)\ds\bigg)_{\big|\S}.
\]
Thus, combining the two relations, there results
\[
 -\i \big(\L \nu(q_0)\big)(t,\cdot) =\big(  (U_t-\mathbb{I}) \, {\mathcal G}^0  \nu(q_0) \big)_{|\S}.
\]
Hence, since
\[
 F_{0,2} (t,\cdot)-F_{0,2} (0,\cdot)=-\big(  (U_t-\mathbb{I}) \, \G  \nu(q_0) \big)_{|\S},
\]
one obtains
\[
 F_{0,2} (t,\cdot)-F_{0,2} (0,\cdot)-\i \big(\L \nu(q_0)\big)(t,\cdot) = -\left(  (U_t-\mathbb{I}) \, \big(\G  \nu(q_0)-{\mathcal G}^0  \nu(q_0)\big) \right)_{\big|\S},
\]
whence
\[
\left\{
\begin{array}{ll}
 \displaystyle (G_2)_{\ell,m}(t)=g_{2,\ell}(t) \, \big(\nu(q_0)\big)_{\ell,m}\qquad\text{with}\\[.4cm]
 \displaystyle g_{2,\ell}(t):=\int_0^{+\infty} r\,\big(\exp^{-\i t r^2}-1\big)J^2_{\ell+1/2} (r)\bigg(\frac{1}{r^2+\l}-\f{1}{r^2}\bigg)  \,dr\\[.4cm]
 \displaystyle \hspace{1.25cm}=-\l\int_0^{+\infty} \f{J^2_{\ell+1/2} (r)}{r\,(r^2+\l)}\big(\exp^{-\i t r^2}-1\big)\,dr.
 \end{array}
 \right.
\]

Now, using \eqref{eq-F1} and changing variables, there results
\[
(F_{0,1})_{\ell,m}(t)=\f{1}{2} \int_0^{+\infty} \exp^{-\i t \ome} {\ome^{1/4}} J_{\ell+1/2} (\sqrt{\ome}) (\widetilde{\phi_0^\l})_{\ell,m} (\sqrt{\ome})\dome=\widehat{g}_{\ell,m}(t),
\]
with
\[
 {g}_{\ell,m}(\ome):=\sqrt{\f{\pi}{2}}\,\mathds{1}_{[0,+\infty)}(\ome){|\ome|^{1/4}} J_{\ell+1/2} (\sqrt{|\ome|}) (\widetilde{\phi_0^\l})_{\ell,m} (\sqrt{|\ome|}),
\]
so that, by\eqref{eq-timef}, \eqref{landau2} and \eqref{fourier1}, 
\begin{align*}
\|F_{0,1} \|_{ H^1 (\R ,H^{2/3}(\S) ) }^2
&=\f{\pi}{2} \somma  \langle \ell \rangle^{2/3}\int_0^{+\infty} \langle \ome\rangle^2 \sqrt{\ome} J^2_{\ell+1/2} (\sqrt{\ome}) \big|(\widetilde{\phi_0^\l})_{\ell,m} (\sqrt{\ome})\big|^2 \dome\\
&=\pi\somma \langle \ell \rangle^{2/3} \int_0^{+\infty} (k^2 +k^6) J^2_{\ell+1/2} (k)  \big|(\widetilde{\phi_0^\l})_{\ell,m} (k)\big|^2 \dkm\\
&\leq c \somma \int_0^{+\infty} (k^2 +k^6) \big|(\widetilde{\phi_0^\l})_{\ell,m} (k)\big|^2 \dkm=c \|\phi_0^\l\|^2_{H^2(\R^3)}.\\
\end{align*}
Since, in addition, $F_{0,1}(0,\cdot)={\phi_0^\l}_{|\S}$, we have that
\[
 \|F_{0,1}(0,\cdot) \|_{ H^1 ([0,T] ,H^{2/3}(\S) ) }^2\leq c T \|\phi_0^\l\|^2_{H^2(\R^3)}
\]
and thus
\begin{equation}
\label{eq-G1stima}
 \|G_1 \|_{ H^1 ([0,T] ,H^{2/3}(\S) ) }\leq c \sqrt{1+T}\, \|\phi_0^\l\|_{H^2(\R^3)},\qquad\forall T>0.
\end{equation}

It is, then, left to estimate $\|G_2 \|_{ H^1 ([0,T] ,H^{2/3}(\S) ) }$. First, one can check by dominated convergence that
\[
 g_{2,\ell}'(t)=\imath\l\int_0^{+\infty} \f{r\,J^2_{\ell+1/2} (r)}{(r^2+\l)}\exp^{-\i t r^2}\,dr,
\]
so that, by \eqref{landau1},
\[
 \|g_{2,\ell}\|_{C^1([0,T])}\leq c_\l \int_0^{+\infty} \frac{r\,J^2_{\ell+1/2} (r)}{r^2+\l}\,dr\leq c_\l <+\infty,\qquad\forall T>0.
\]
Hence, arguing as in Part (i) and using \eqref{schauder}, there results that $\|G_2 \|_{ C^1 ([0,T] ,H^{3/2}(\S) ) }\leq c_\l\|q_0\|^{2\sigma+1}_{H^{3/2}(\S)}$, for every $T>0$, and thus, as
\[
  C^1 \big([0,T] ,H^{3/2}(\S) \big) \hookrightarrow H^1\big([0,T],H^{3/2}(\S)\big) \hookrightarrow  H^1\big([0,T],H^{2/3}(\S)\big),
\]
one obtains
\[
 \|G_2 \|_{ H^1 ([0,T] ,H^{2/3}(\S) ) }\leq c_\l\|q_0\|^{2\sigma+1}_{H^{3/2}(\S)},\qquad\forall T>0,
\]
which combined with \eqref{eq-G1stima} completes the proof.

\emph{Part (iii): proof of \eqref{fase1}}. Combining \eqref{fase2} with
\[
 H^{3/2}(\S)\hookrightarrow H^{2/3}(\S)\qquad\text{and}\qquad H^\al\big([0,T],H^{2/3}(\S)\big)\hookrightarrow  H^1\big([0,T],H^{2/3}(\S)\big),
\]
one can see that it is sufficient to estimate $\| \L \nu(q_0) \|_{ H^{\alpha}([0,T],H^{3/2}(\S) )}$. First we note that, for any function $g\in H^{3/2}(\S)$,  $g\equiv\chi_{[0,T]}g$ on $[0,T]\times\S$. Moreover, since
\[
\widehat{ \chi_{[0,T]}}(\ome) =\f{\imath}{\sqrt{2\pi}} \frac{\exp^{-\imath\ome T}-1}{\ome},
\]
$\chi_{[0,T]} g \in  H^{\gamma}\big([0,T],H^{3/2}(\S)\big) $, for every $\gamma\in\big(0,\f{1}{2}\big)$. Hence, by Corollary \ref{cor-reg2},
\beq \label{intermedia}
\| \L g \|_{ H^{\beta+1/7}([0,T],H^{3/2}(\S) )} \leq c_{\beta,T}\|g\|_{H^{3/2} (\S)},\qquad\forall \beta\in\big(0,\tf{1}{2}\big),
\eeq
and thus, setting $g=\nu(q_0)$, there results
\beq \label{intermedia2}
\| \L \nu(q_0) \|_{ H^{\alpha}([0,T],H^{3/2}(\S) )} \leq c_{\alpha,T} \|q_0\|_{H^{3/2}(\S)}^{2\sigma+1},\qquad\forall \alpha\in\big(0,\tf{9}{14}\big),
\eeq
which completes the proof.
\end{proof}


\section{Local well--posedness: proof of Theorem \ref{main} -- item (i)}
\label{sec-local}

The main tool for the proof of local well--posedness of \eqref{eq-cauchy} in $\D(\h)$ is establishing existence and uniqueness of the solutions of \eqref{eq-charge} with a suitable regularity.

\begin{prop} \label{carica}
Let $\beta\in\R$, $\sigma\geq1/2$ and $\psi_0=\phi_{0}^\l-\G \nu(q_0) \in \D(\h)$. Then:
\begin{itemize}
 \item[(i)] there exists $T_0>0$ for which there is a unique solution of \eqref{eq-charge} in $C^0\big([0,T_0], H^{3/2}(\S) \big)$;
 \item[(ii)] if $q$ is the unique solution of \eqref{eq-charge} in $C^0\big([0,T], H^{3/2}(\S) \big)$, for some $T>0$, then $q\in H^1\big([0,T],L^{2}(\S)\big)$.
\end{itemize}
\end{prop}

\begin{proof}
It is convenient to divide the proof in two steps.

\emph{Step (i).} It is sufficient to show that, for $T>0$ sufficiently small, the map
\begin{equation}
\label{eq-mappa}
\LL (q) := -\i \L \nu (q)  + F_0,
\end{equation}
is a contraction in
\[
X:=\left\{ q\in  L^\infty\big([0,T], H^{3/2}(\S) \big): \| q\|_{  L^\infty([0,T], H^{3/2}(\S) ) }\leq R \right\} 
\]
for a fixed $R> \| F_0 \|_{ L^\infty([0,T], H^{3/2}(\S) ) }$, with respect to a proper metric that make $X$ complete. Continuity can be easily established afterwards combining \eqref{eq-final} and \eqref{sourceinfinito}.

Using again \eqref{eq-final}, \eqref{sourceinfinito} and \eqref{schauder}, we find that
\[
\| \LL( q)\|_{  L^\infty([0,T], H^{3/2}(\S) ) } \leq c\, T^{1/3} R^{2\sigma+1} + \| F_0 \|_{  L^\infty([0,T], H^{3/2}(\S) ) },\qquad\forall q\in X,
\]
and thus $X$ is preserved by $\LL$ for $T$ small enough. It is left to discuss contractivity with respect to a suitable metric. Consider the $L^\infty\big([0,T],L^2(\S)\big)$--one. It clearly makes $X$ complete since an $L^\infty\big([0,T],L^2(\S)\big)$--Cauchy sequence in $X$ converges in $L^\infty\big([0,T],L^2(\S)\big)$ to a limit which has to belong to $X$ since the sequence
is also weakly$^*$ convergent in $L^\infty([0,T], H^{3/2}(\S) )$ by Banach–Alaoglu. Then, fix $q_1,\, q_2\in \big(X,\|\cdot\|_{L^\infty([0,T],L^2(\S))}\big)$. Since
\begin{equation}
\label{eq-banaleutile}
 \LL(q_1)-\LL(q_2)=\L\big(\nu(q_2)-\nu(q_1)\big),
\end{equation}
combining \eqref{eq-final}, \eqref{eq-nonlin} and \eqref{schauder}, one can find that
\begin{align*}
\| \LL (q_1) - \LL (q_2) \|_{  L^\infty([0,T], L^2(\S) ) } 
& \leq c T^{1/3}  \|  \nu( q_2) - \nu(q_1)  \|_{  L^\infty([0,T], L^2(\S) ) } \\
& \leq c T^{1/3} \, R^{2\sigma}  \|   q_1 - q_2  \|_{ L^\infty([0,T], L^2(\S) )},
\end{align*}
and thus $\LL$ is contractive again for $T>0$ sufficiently small.

\emph{Step (ii).} Here we use a standard iterative bootstrap argument. Let $q$ be a solution of \eqref{eq-charge} in $ C^0\big([0,T], H^{3/2}(\S) \big)$. As a consequence $q \in L^2([0,T], L^{2}(\S) ) $, as well, and satisfies 
\beq \label{boot}
q = -\i \L \nu (q)+F_0.
\eeq
Since, arguing as before, $\nu(q)\in C^0\big([0,T], H^{3/2}(\S)\big )\subset L^2\big([0,T], L^2(\S) \big)$, by item (i) of Corollary \ref{cor-reg2} there results that $\L \nu(q ) \in H^{1/7} \big([0,T], L^{2}(\S) \big) $. In addition, as from \eqref{fase1} $F_0 \in H^{1/7} \big([0,T], L^{2}(\S) \big) $, by \eqref{boot} we have that $q\in H^{1/7} \big([0,T], L^{2}(\S) \big)$, which concludes the first iteration. Now, since by \eqref{eq-nonlin}
\[
\| \nu(q)(t) -\nu(q)(s) \|_{L^2(\S) } \leq \| q\|_{  C^0([0,T], H^{3/2}(\S) )}^{2\sigma}\| q(t) - q(s) \|_{ L^2 (\S) },\qquad\forall t,\,s\in[0,T],
\]
$\nu(q) \in H^{1/7} \big([0,T], L^{2}(\S) \big)$. Hence, arguing as before, one obtains $q\in H^{2/7} \big([0,T], L^{2}(\S) \big)$ and, by means of two further analogous iterations, one gets $q\in H^{4/7} \big([0,T], L^{2}(\S) \big)$.

Here, a new issue arises since both $\nu(q) \in H^{4/7} \big([0,T], L^{2}(\S) \big)$, preventing the use item (i) of Corollary \ref{cor-reg2}, and $\big\|\big(\nu(q)\big)(t,\cdot)\big\|_{L^2(\S)}\neq0$, preventing the use item (ii) of Corollary \ref{cor-reg2}. However, one can rewrite \eqref{boot} as
\beq
\label{boot2}
q-q_0 = -\i \L \big( \nu (q)-\nu(q_0) \big) + F_0-q_0 -\i \L \nu(q_0).
\eeq
Thus $\nu (q)-\nu(q_0)\in H^{4/7} \big([0,T], L^{2}(\S) \big)$ and $\big\|\big(\nu(q)\big)(t,\cdot)-\nu(q_0)\big\|_{L^2(\S)}=0$, so that one may exploit item (ii) of Corollary \ref{cor-reg2} getting $\L \big( \nu (q)-\nu(q_0) \big)\in H^{5/7} \big([0,T], L^{2}(\S) \big)$. Moreover, as by \eqref{fase2} we have $F_0-q_0 -\i \L \nu(q_0) \in H^{5/7} \big([0,T], L^{2}(\S) \big) $, there results $q-q_0\in H^{5/7} \big([0,T], L^{2}(\S) \big)$, whence $q\in H^{5/7} \big([0,T], L^{2}(\S) \big)$. Iterating such procedure twice more one finally gets the claim. 
\end{proof}

\begin{rem}
 Notice that, on any interval $[0,T]$ where \eqref{eq-charge} admits a solution, the nonlinearity enjoys the same regularity of $q$, that is
\beq
\label{eq-nureg}
\nu(q) \in C^0\big([0,T], H^{3/2}(\S) \big) \cap  H^1\big([0,T],L^2(\S)\big).
\eeq
\end{rem}

\begin{rem}
\label{sigma}
Step (i) of the proof above clearly shows the reasons why we have to require $\sigma\geq1/2$. Indeed, on the one hand, as we saw by Section \ref{sec-opLambda}, we are able to manage the operator $\L$ only on functions which are $H^\mu(\S)$--regular (with $\mu\geq0$) in space; while, on the other hand, we have to use \eqref{schauder} with $\mu>1$, which requires $\sigma\geq1/2$. This is a major difference with point delta models, where one may avoid the use of \eqref{schauder} by applying Fixed Point theorem in $L^\infty([0,T])$, since the spacial part is absent and $\L$ is a time--only operator.
\end{rem}

Now, before presenting the proof of item (i) of Theorem \ref{main}, we have to introduce a further auxiliary result. To this aim, denote by $V_t$ the restriction of the operator $U_t$ to functions defined on $\S$. As a consequence, its integral kernel is given by $\U (t, \x-\y)_{\big||\y|=1}$. Note also that we
\begin{equation}
\label{eq-group}
 V_{t+s}= U_t \, V_s= U_s \, V_t,
\end{equation}
by the group properties of $U_t$.

\begin{prop} \label{ricostruzione}
Let  $g \in H^1\big([0,T], L^2(\S)\big) $ and $\displaystyle h(t,\x) := \int_0^t\big( V_{t-s}  g(s)\big)(\x) \ds$. Then 
\begin{equation}
\label{eq-primastima}\|h\|_{ C^0([0,T], L^2 (\R^3) )} \lesssim_T \|g \|_{ L^2([0,T], L^2(\S))}.
\end{equation}
\end{prop}

\begin{proof}
First, using \eqref{fourier1} and \eqref{eq-rinota}, we obtain
\[
\F[h](t,\k) = \exp^{-\i |\k|^2 t} \somma (-\i)^\ell   Y_{\ell,m} \lf( \f{\k}{|\k|} \ri) \int_0^t \exp^{\i |\k|^2 s} \f{J_{\ell+1/2}(|\k|) }{\sqrt |\k|} g_{\ell,m} (s) \ds,
\]
so that, for every $t\in[0,T]$,
\begin{align}
\| h(t)\|_{L^2 (\R^3)}^2 & = \somma \int_0^{+\infty} k^2 \lf|  \f{J_{\ell+1/2}(k) }{\sqrt k} \int_0^t \exp^{\i k^2 s}  g_{\ell,m} (s) \ds \ri|^2 \, dk \nonumber  \\
& = \somma \int_0^{+\infty} k J_{\ell+1/2}^2 (k)\lf|   \int_0^t \exp^{\i k^2 s}  g_{\ell,m} (s) \ds \ri|^2 \, dk \nonumber\\
& =  \somma    \int_0^t\int_0^t  g_{\ell,m}^* (s)  g_{\ell,m} (s')   \lf( \int_0^\infty k J_{\ell+1/2}^2 (k)  \exp^{-\i k^2 (s-s') } \, dk \ri) \ds \ds' \nonumber \\
& =  \somma    \int_0^t g_{\ell,m}^* (s) \int_0^t O_\ell (s'-s) g_{\ell,m} (s')     \ds' \ds. \nonumber\\
& \leq \somma \|g_{\ell,m}\|_{L^2(0,T)}\bigg\|\int_0^t O_\ell (s'-(\cdot)) g_{\ell,m} (s')     \ds'\bigg\|_{L^2(0,T)},\label{assoluto}
\end{align}
with
\begin{equation}
\label{eq-Oelle}
 O_\ell(\tau):=\int_0^{+\infty}kJ_{\ell+1/2}^2(k)\exp^{\imath k^2\tau}\dkm.
\end{equation}
Moreover, as we show in Appendix \ref{app-formula},
\beq \label{eq-integralbessel}
O_\ell (\tau) = \frac{\exp^{\i(\ell+3/2)\frac{\pi}{2} } }{2\tau} \exp^{-\frac{\i}{2\tau} } J_{\ell+1/2} \lf( \f{1}{2\tau} \ri),
\eeq
which implies, by \eqref{landau1},
\beq 
|O_\ell (\tau) | \leq \f{c}{|\tau|^{2/3} }.
\eeq
Hence, using the Schur test between $L^2(0,T)$ and itself (e.g., \cite{HS-78})
\[
 \bigg\|\int_0^t O_\ell (s'-(\cdot)) g_{\ell,m} (s')     \ds'\bigg\|_{L^2(0,T)}\leq c_T \,\|g_{\ell,m}\|_{L^2(0,T)}
\]
and therefore, combining with \eqref{assoluto},
\[
\| h(t)\|_{L^2 (\R^3)}^2 \leq c_T  \somma \| g_{\ell, m} \|_{L^2(0,T)}^2  = \|g \|_{ L^2([0,T], L^2(\S))}^2,
\] 
which proves \eqref{eq-primastima} for the $L^\infty \big([0,T], L^2 (\R^3) \big)$--norm. It is, then, left to prove that $h \in C^0\big([0,T], L^2 (\R^3) \big)$. However, this can be easily checked arguing as before and using dominated convergence to prove that $\|h(t+\ep) -f (t) \|_{L^(\R^3) }\to 0$, as $\ep\to0$, for every $t\in[0,T]$.
\end{proof}

\begin{proof}[Proof of Theorem \ref{main} -- item (i)]
In view of Proposition \ref{carica}, it is sufficient to prove that, given a solution $q$ of \eqref{eq-charge} in $C^0\big([0,T],H^{3/2}(\S)\big)$, the function $\psi$ defined by \eqref{eq-duhamel} satisfies \eqref{eq-regularity} and \eqref{eq-cauchy} in $L^2(\R^3)$ for every $t\in[0,T]$. Indeed, arguing as in the linear case, one can prove that any solution of \eqref{eq-cauchy} satisfying \eqref{eq-regularity} has to fulfill \eqref{eq-duhamel}--\eqref{eq-charge} as well, and thus uniqueness for \eqref{eq-cauchy} is equivalent to uniqueness for \eqref{eq-charge}, which is proved by Proposition \ref{carica}.

As a first step, we recall that by Stone's Theorem
\begin{equation}
\label{eq-stone}
 \f{d}{ds}U_{-s}=-\imath\Delta U_{-s},
\end{equation}
so that
\begin{equation}
\label{eq-consstone}
\frac{d}{ds} \lf( -\i (-\Delta+\l)^{-1} \exp^{\i \l s}U_{-s} \ri) = \exp^{\i \l s}U_{-s}.
\end{equation}
Hence, combining \eqref{eq-duhamel} with \eqref{eq-consstone}, \eqref{eq-group}, the properties of the Green's potentials, the commutation between $U_t$ and $(-\Delta+\l)^{-1}$ and the integration by parts for operator-valued functions, there results
\begin{align*}
\psi (t) &= U_t \lf(\phi_0^\l -\G \nu(q_0) \ri) -\i \int_0^t V_{t-s} \nu \big(q(s)\big) \ds \\
&= U_t \lf(\phi_0^\l -\G  \nu(q_0) \ri) -\i U_t \int_0^t \exp^{\i \l s}V_{-s}\, \exp^{-\i \l s}\nu \big(q(s)\big) \ds \\
& = U_t \lf(\phi_0^\l -\G  \nu(q_0) \ri) -\i U_t \int_0^t \frac{d}{ds} \lf( -\i (-\Delta+\l)^{-1} \exp^{\i \l s}V_{-s} \ri) \, \exp^{-\i \l s}\nu \big(q(s)\big) \ds \\
&=  U_t \phi_0^\l -\G  \nu(q(t)) +(-\Delta+\l)^{-1} \int_0^t V_{t-s}  \lf[ -\i\l \nu \big(q(s)\big) + \partial_s\nu\big(q(s)\big)\ri] \ds\\
&=\phi^\l (t)-\G  \nu(q(t)),
\end{align*}
where
\beq \label{parteregolare}
\phi^\l (t) := U_t \varphi_0^\l +(-\Delta+\l)^{-1} \int_0^t V_{t-s}  \lf[ -\i\l \nu \big(q(s)\big) + \partial_s\nu\big(q(s)\big)\ri] \ds,
\eeq
(we omitted the $\x$ dependence and used $\partial_s$ for partial derivatives with respect to $s$ for the sake of simplicity). Now, again by the Stone's theorem $U_{(\cdot)} \phi_0^\l \in C^0\big([0,T], H^2(\R^3) \big)$. Moreover, by \eqref{eq-nureg}, we have that $\nu(q),  \partial_s\nu(q) \in L^2([0,T], L^2(\S))$ and, thus, by Proposition \ref{ricostruzione} the boundedness of $(-\Delta+\l)^{-1}:L^2(\R^3)\to H^{2}(\R^3)$ and \eqref{parteregolare}, there results that $\phi^\l\in C^0\big([0,T], H^2(\R^3) \big)$. Hence, since by \eqref{eq-duhamel}--\eqref{eq-charge} it is straightforward that $\psi_{|\S}\equiv q$, and since in view of Remark \ref{rem-opnorm}
\[
 \|\psi(t)\|_{\D(\h)}\leq \|\phi^\l(t)\|_{H^2(\R^3)}+(1+\l)\big\|\G\nu\big(q(t)\big)\big\|_{L^2(\R^3)},
\]
in order to get $\psi\in C^0\big([0,T], \D (\h)\big)$ it is sufficient to show that $\G\nu(q)\in C^0\big([0,T], L^2(\R^3) \big)$. However, this can be easily obtained as by Proposition \ref{carica}
\[
 \big\|\G\nu\big(q(t+h)\big)-\G\nu\big(q(t)\big)\big\|_{L^2(\R^3)}\leq \f{c}{\l}\|q\|_{C^0([0,T],H^{3/2}(\S))}^{2\sigma}\|q(t+h)-q(t)\|_{H^{3/2}(\S)}.
\]

Then, observing that the above regularity and \eqref{eq-duhamel} imply $\psi(0)\equiv\psi_0$, it is left to prove that $\psi\in C^{1}\big([0,T],L^2(\R^3)\big)$ and the the equation in \eqref{eq-cauchy} is satisfied in $L^2(\R^3)$, for all $t\in[0,T]$. However, straightforward calculations on  \eqref{parteregolare} yield
\[
\i \f{\partial\phi^\l}{\partial t} = -\Delta \phi^\l+\l\G\nu(q) +\i  \G\f{\partial\nu(q)}{\partial t},
\]
and thus, since $\f{\partial\psi}{\partial t}=\f{\partial\phi^\l}{\partial t}-\G\f{\partial\nu(q)}{\partial t}$, the regularity proved above implies that
\[
 \psi  \in C^1 \big([0,T], L^2(\S)\big),
\]
while from \eqref{eq-nonact} one obtains that
\[
 \imath \f{\partial\psi}{\partial t}=\h\psi.
\]
\end{proof}


\section{Conservation laws: proof of Theorem \ref{main} -- item (ii)}
\label{sec-cons}

This section is devoted to the proof of the conservation laws associated with \eqref{eq-cauchy}: the conservation of the mass, i.e. \eqref{massa}, and the conservation of the energy, i.e. \eqref{energia}. 

Clearly, in the following we tacitly assume that $\psi$ is the function defined by \eqref{eq-duhamel} which satisfies item (i) of Theorem \ref{main} on a fixed interval $[0,T]$, and that $q\equiv\psi_{|\S}$ satisfies \eqref{eq-charge} with the regularity provided by Proposition \ref{carica}. In such a way, we can neglect in the following proof any regularity issue, since the features of $\psi$ and $q$ make all the steps rigorous. In addition, we recall that all the scalar products below have to be meant as antilinear in the first component.

\begin{proof}[Proof of Theorem \ref{main} --  item (ii)]
 We divide the proof in two parts.
 
 \emph{Part (i): proof of \eqref{massa}.} Since $\psi\in C^1\big([0,T],L^2(\R^3)\big)$, we have that $M(t):=M\big[\psi(t)\big]$ is of class $C^1$.  Let us prove, then, that
\beq \label{massa1}
 M'(t)=0,\qquad\forall t\in[0,T].
\eeq
By definition
\[
 M'(t)= 2\mathrm{Re}\big\{\langle \psi(t) , \partial_t \psi(t) \rangle_{L^2(\R^3)}\big\}.
\]
As by \eqref{eq-cauchy} and \eqref{eq-nonact}
\[
\i \f{\partial\psi}{\partial t} = (-\Delta +\l )\phi^\l - \l \psi,
\]
using the properties of the Green's potential, \eqref{BC} and Proposition \ref{pro-BC} 
\begin{align*}
M'(t)
& = 2\text{Re}\left\{ \big\langle \psi(t) , -\i (-\Delta +\l )\phi^\l (t)  \big\rangle_{L^2(\R^3)}\right\} \\[.2cm]
& = 2\text{Re} \left\{\big\langle \G \nu\big(q(t)\big)  , \i (-\Delta +\l )\phi^\l (t)  \big\rangle_{L^2(\R^3)}\right\} \\[.2cm]
& = 2\text{Re} \left\{\imath\,\big\langle  \nu\big(q(t)\big)  , \phi^\l_{|\S} (t)  \big\rangle_{L^2(\S)}\right\} \\[.2cm]
& = 2\text{Re} \left\{\imath\,\big\langle  \nu\big(q(t)\big)  ,   q(t) + T^\l \nu( q(t)) \big\rangle_{L^2(\S)}\right\}\\[.2cm]
& = 2\text{Re} \left\{\imath\,\big\langle  \nu\big(q(t)\big)  ,  T^\l \nu( q(t)) \big\rangle_{L^2(\S)}\right\}.
\end{align*}
Finally, since by \eqref{eq-green} $G^\l$ is real--valued and even, $\big\langle  \nu\big(q(t)\big)  ,  T^\l \nu\big( q(t)\big) \big\rangle_{L^2(\S)}$ is equal to its complex conjugate and thus we obtain \eqref{massa1}.

\emph{Part (ii): proof of \eqref{energia}.} Preliminarily, we note that by Remark \ref{rem-altramap} the energy is well-defined for every $t\in[0,T]$. In contrast to Part (i), here we prove $E[\psi(t)]=:E(t)=E(0):=E[\psi_0]$, for every $t\in[0,T]$, by a direct inspection. First, using \eqref{eq-duhamel} and \eqref{eq-group} and recalling the definition of $V_t$ given after Remark \ref{sigma} and the fact that Sobolev homogeneous norms (denoted below by $[\cdot]_{\sn{H}^m(\R^3)}$) are invariant under the action of the free propagator $U_t$, we get 
\[
\|\nabla\psi(t)\|^2_{L^2(\R^3)} = \lf[\psi_0 -\i \int_0^t   V_{-s} \nu\big(q(s)\big) \, {\rm d }s   \ri]^2_{\sn{H}^1(\R^3)}.
\]
Then, expanding the square and integrating by parts in $\R^3$, (as the boundary term vanishes) we get
\begin{multline}
\label{eq-gradiente}
 \|\nabla\psi(t)\|^2_{L^2(\R^3)}\\
 = \|\nabla\psi_0\|^2_{L^2(\R^3)} \underbrace{- 2{\rm Re}\bigg\{ \i \,\big\langle \psi_0 , \int_0^t   -\Delta V_{-s} \nu\big(q(s)\big) \ds \big\rangle_{L^2(\R^3)} \bigg\}}_{=:A_1}+\underbrace{
 \lf[\int_0^t   V_{-s} \nu\big(q(s)\big) \ds  \ri]^2_{\sn{H}^1(\R^3)}}_{=:A_2}.
\end{multline}
Let us start by discussing $A_1$. From \eqref{eq-stone} and integration by parts for operator--valued functions there results
\[
 \int_0^t   -\Delta V_{-s} \nu\big(q(s)\big) \ds=-\imath V_{-t}\nu\big(q(t)\big)+\imath\nu(q_0)\delta_{\S}+\imath\int_0^tV_{-s}\partial_s\nu\big(q(t)\big)\ds,
\]
so that, recalling that $\psi_{|\S}\equiv q$ and noting that by definition $V_{-t}^*g=(U_tg)_{|\S}$,
\begin{multline*}
\i \,\big\langle \psi_0 , \int_0^t   -\Delta V_{-s} \nu\big(q(s)\big) \ds \big\rangle_{L^2(\R^3)} \\
=  \big\langle (U_t\psi_0)_{|\S} , \nu\big(q(t)\big) \big\rangle_{L^2(\S)} - \beta\|q_0\|_{L^{2\sigma+2}(\S)}^{2\sigma+2}  -\int_0^t  \big\langle (U_s\psi_0)_{|\S} , \partial_s \nu\big(q(s)\big) \big\rangle_{L^2(\S)} \ds.
\end{multline*}
As a consequence,
\begin{multline}
 \label{eq-A_1}
 A_1=2\beta\|q_0\|_{L^{2\sigma+2}(\S)}^{2\sigma+2}-2{\rm Re}\left\{\big\langle \nu\big(q(t)\big), (U_t\psi_0)_{|\S}  \big\rangle_{L^2(\S)}\right\}+\\[.2cm]
 +2{\rm Re}\bigg\{\int_0^t  \big\langle (U_s\psi_0)_{|\S} , \partial_s \nu\big(q(s)\big) \big\rangle_{L^2(\S)} \ds\bigg\}.
\end{multline}
On the other hand, concerning $A_2$, we first see by the Fubini theorem that
\[
 A_2=2{\rm Re}\bigg\{\int_0^t\int_0^s\big\langle V_{-s}\nu\big(q(s)\big),V_{-\tau}\nu\big(q(\tau)\big)\big\rangle_{\sn{H}^1(\R^3)}\dtau\ds\bigg\}.
\]
Moreover, using again \eqref{eq-stone} and integration by parts (both in $\R^3$ and for operator--valued functions) and observing that \eqref{eq-kernI} and \eqref{eq-I} yield $V_{-t}^*V_{-s}=I_{t-s}$, there results
\begin{align*}
 \int_0^t\int_0^s\big\langle V_{-s}\nu\big(q(s)\big),V_{-\tau}\nu\big(q(\tau)\big)\big\rangle_{\sn H^1(\R^3)}\dtau\ds & \\[.2cm]
 & \hspace{-4cm}=\imath\int_0^t\big\langle(\partial_s V_{-s})\nu\big(q(s)\big),\int_0^sV_{-\tau}\nu\big(q(\tau)\big)\dtau\big\rangle_{L^2(\R^3)}\ds\\[.2cm]
 & \hspace{-3.5cm}+\imath\big\langle V_{-t}\nu\big(q(t)\big),\int_0^tV_{-s}\nu\big(q(s)\big)\ds\big\rangle_{L^2(\R^3)}+\\[.2cm]
 & \hspace{-3.5cm}-\imath\int_0^t\big\langle V_{-s}\partial_s\nu\big(q(s)\big),\int_0^sV_{-\tau}\nu\big(q(\tau)\big)\dtau\big\rangle_{L^2(\R^3)}\ds+\\[.2cm]
 & \hspace{-3.5cm}-\imath\int_0^t\big\langle V_{-s}\nu\big(q(s)\big),V_{-s}\nu\big(q(s)\big)\big\rangle_{L^2(\R^3)}\ds\\[.2cm]
 & \hspace{-4cm}=\imath\,\big\langle\nu\big(q(t)\big),\int_0^tI_{t-s}\nu\big(q(s)\big)\ds\big\rangle_{L^2(\S)}-\imath\int_0^t\big\|V_{-s}\nu\big(q(s)\big)\big\|_{L^2(\R^3)}^2\ds+\\[.2cm]
 & \hspace{-3.5cm}-\imath\int_0^t\int_0^s\big\langle \partial_s\nu\big(q(s)\big),I_{s-\tau}\nu\big(q(\tau)\big)\big\rangle_{L^2(\S)}\dtau\ds
\end{align*}
and thus
\begin{multline}
 \label{eq-A_2}
 A_2=2{\rm Re}\bigg\{\imath\,\big\langle\nu\big(q(t)\big),\int_0^tI_{t-s}\nu\big(q(s)\big)\ds\big\rangle_{L^2(\S)}\bigg\}+\\[.2cm]
 -2{\rm Re}\bigg\{\imath\int_0^t\int_0^s\big\langle \partial_s\nu\big(q(s)\big),I_{s-\tau}\nu\big(q(\tau)\big)\big\rangle_{L^2(\S)}\dtau\ds\bigg\}.
\end{multline}
Summing up, combining \eqref{eq-gradiente}, \eqref{eq-A_1} and \eqref{eq-A_2},
\begin{align*}
 \|\nabla\psi(t)\|^2_{L^2(\R^3)} =\: & \|\nabla\psi_0\|^2_{L^2(\R^3)}+2\beta\|q_0\|_{L^{2\sigma+2}(\S)}^{2\sigma+2}+\underbrace{2{\rm Re}\left\{\big\langle \nu\big(q(t)\big), -(U_t\psi_0)_{|\S}  \big\rangle_{L^2(\S)}\right\}}_{=:B_1}+\\
 & +\underbrace{2{\rm Re}\bigg\{\big\langle\nu\big(q(t)\big),\imath\int_0^tI_{t-s}\nu\big(q(s)\big)\ds\big\rangle_{L^2(\S)}\bigg\}}_{=:B_2}+\\
 & +\underbrace{2{\rm Re}\bigg\{\int_0^t  \big\langle (U_s\psi_0)_{|\S} , \partial_s \nu\big(q(s)\big) \big\rangle_{L^2(\S)} \ds\bigg\}}_{=:B_3}+\\
 & \underbrace{-2{\rm Re}\bigg\{\imath\int_0^t\int_0^s\big\langle \partial_s\nu\big(q(s)\big),I_{s-\tau}\nu\big(q(\tau)\big)\big\rangle_{L^2(\S)}\dtau\ds\bigg\}}_{=:B_4}.
\end{align*}
Now, using \eqref{eq-charge}, \eqref{source} and \eqref{eq-Lambda}, we find that
\begin{align*}
 B_1+B_2 & =-2\beta\|q(t)\|_{L^{2\sigma+2}(\S)}^{2\sigma+2}\\[.2cm]
 B_3+B_4 & =2{\rm Re}\bigg\{\int_0^t\big\langle\partial_s\nu\big(q(s)\big),q(s)\big\rangle_{L^2(\S)}\ds\bigg\},
\end{align*}
which entails, integrating by parts in $s$,
\[
 \|\nabla\psi(t)\|^2_{L^2(\R^3)} = \|\nabla\psi_0\|^2_{L^2(\R^3)}-2{\rm Re}\bigg\{\int_0^t\big\langle\partial_sq(s),\nu\big(q(s)\big)\big\rangle_{L^2(\S)}\ds\bigg\}.
\]
Finally, as
\[
 2{\rm Re}\left\{\big\langle\partial_sq(s),\nu\big(q(s)\big)\big\rangle_{L^2(\S)}\right\}=\frac{\beta}{\sigma+1}\frac{d}{ds}\|q(t)\|_{L^{2\sigma+2}(\S)}^{2\sigma+2},
\]
suitably rearranging terms, one recovers \eqref{energia}.
\end{proof}


\section{Global well--posedness: proof of Theorem \ref{main2}}
\label{sec-global}

In this section we discuss global well--posedness of \eqref{eq-cauchy} in $\D(\h)$; that is, we discuss under which assumptions one can prove that the parameter $T^*$, defined by \eqref{eq-tmax}, is equal to $+\infty$. However, in view of the arguments developed in Section \ref{sec-local}, one can see that
\begin{equation}
\label{eq-tmaxeq}
 T^*=\sup\left\{T>0:\text{\eqref{eq-charge} admits a unique solution in }C^0\big([0,T],H^{3/2}(\S)\big)\right\},
\end{equation}
so that global well--posedness of \eqref{eq-cauchy} in $\D(\h)$ turns out to be equivalent to global well--posedness of \eqref{eq-charge} in $H^{3/2}(\S)$, which is the issue that we actually address below.


\subsection{Global well--posedness for small data: proof of Theorem \ref{main2} -- item (i)}
\label{subsec-globalsmall}

The former strategy to prove a global--well posedness result is to find a contraction argument in $C^0\big([0,+\infty),H^{3/2}(\S)\big)$ analogous to the one used to prove item (i) of Proposition \ref{carica}.

\begin{proof}[Proof of Theorem \ref{main2} -- item (i)]
It is sufficient to show that, whenever $\psi_0=\phi_0^\la-\G\nu(q_0)$ is such that $\|\phi_0^\la\|_{H^2(\R^3)}$ is small enough, the map $\LL$ defined by \eqref{eq-mappa} is a contraction in
\[
Y:=\left\{ q\in  L^\infty \big([0,+\infty), H^{3/2}(\S)\big):\| q\|_{  L^\infty([0,+\infty),  H^{3/2}(\S) ) }\leq L\right\}
\]
for some suitable $L>0$ (that we fix below), with respect to a proper metric that make $Y$ complete. As pointed out in the proof of Proposition \ref{carica}, continuity can be easily established ex--post by using \eqref{eq-final} and \eqref{sourceinfinito}.

As a first step, we have to prove that $\LL$ maps $Y$ into itself. Fix, then $q\in Y$. Preliminarily we note that
\[
 \big\|\big(\L\nu(q)\big)(t)\big\|_{L^2(\S)}\leq\int_0^t\big\|I_s\nu\big(q(t-s)\big)\big\|_{L^2(\S)}\ds,\qquad\forall t\geq0.
\]
On the one hand, when $t\in[0,1]$, using \eqref{disp} with $r=p=2$ and \eqref{schauder}, there results
\begin{align}
\label{eq-tpiccoli}
 \big\|\big(\L\nu(q)\big)(t)\big\|_{L^2(\S)}\leq & c\int_0^t\frac{1}{s^{2/3}}\big\|\nu\big(q(t-s)\big)\big\|_{L^2(\S)}\ds\leq c\int_0^t\frac{1}{s^{2/3}}\|q(t-s)\|_{H^{3/2}(\S)}^{2\sigma+1}\ds\nonumber\\[.2cm]
 \leq & c\|q\|_{L^\infty([0,+\infty),H^{3/2}(\S))}^{2\sigma+1}\int_0^t\frac{1}{s^{2/3}}\ds\leq c\|q\|_{L^\infty([0,+\infty),H^{3/2}(\S))}^{2\sigma+1}.
\end{align}
On the other hand, when $t>1$
\begin{equation}
\label{eq-tgrandi}
 \big\|\big(\L\nu(q)\big)(t)\big\|_{L^2(\S)}\leq\int_0^1\big\|I_s\nu\big(q(t-s)\big)\big\|_{L^2(\S)}\ds+\int_1^t\big\|I_s\nu\big(q(t-s)\big)\big\|_{L^2(\S)}\ds.
\end{equation}
Thus, since arguing as before 
\[
 \int_0^1\big\|I_s\nu\big(q(t-s)\big)\big\|_{L^2(\S)}\ds\leq c\|q\|_{L^\infty([0,+\infty),H^{3/2}(\S))}^{2\sigma+1}
\]
and since using \eqref{disp} with $r=\infty,\,p=1$ (and \eqref{schauder})
\begin{align*}
\int_1^t\big\|I_s\nu\big(q(t-s)\big)\big\|_{L^2(\S)}\ds\leq & c\int_1^t\big\|I_s\nu\big(q(t-s)\big)\big\|_{L^\infty(\S)}\ds\leq c\int_1^t\frac{1}{s^{3/2}}\big\|\nu\big(q(t-s)\big)\big\|_{L^1(\S)}\ds\\[.2cm]
\leq & c\int_1^t\frac{1}{s^{3/2}}\big\|\nu\big(q(t-s)\big)\big\|_{L^2(\S)}\ds\leq c\int_1^t\frac{1}{s^{3/2}}\|q(t-s)\|_{H^{3/2}(\S)}^{2\sigma+1}\ds\\[.2cm]
 \leq & c\|q\|_{L^\infty([0,+\infty),H^{3/2}(\S))}^{2\sigma+1}\int_1^t\frac{1}{s^{3/2}}\ds\leq c\|q\|_{L^\infty([0,+\infty),H^{3/2}(\S))}^{2\sigma+1},
\end{align*}
\eqref{eq-tpiccoli} and \eqref{eq-tgrandi} (and \eqref{infi}) yield
\begin{equation}
\label{eq-stimaL2}
 \| \L \nu (q) \|_{  L^\infty([0,+\infty,  L^2(\S) ) }\leq c\|q\|_{L^\infty([0,+\infty),H^{3/2}(\S))}^{2\sigma+1}.
\end{equation}
Now, combining \eqref{eq-stimaL2} with \eqref{eq-Lambdadec} and with the representation of the Laplace--Beltrami operator in spherical harmonics (i.e., \eqref{eq-diagLB}) and arguing as before, we also find that
\begin{multline}
\label{eq-stimasemi}
\| \L \nu (q) \|_{  L^\infty([0,+\infty),  \sn H^{3/2}(\S) ) }\\
=\| \L  (-\Delta_{\S})^{3/4} \nu (q) \|_{  L^\infty([0,+\infty),  L^2(\S) ) }\leq c\|q\|_{L^\infty([0,+\infty),H^{3/2}(\S))}^{2\sigma+1}.
\end{multline}
As a consequence, from \eqref{eq-stimaL2} and \eqref{eq-stimasemi}
\[
\| \LL (q) \|_{ L^\infty([0,+\infty),  H^{3/2}(\S) )} \leq \| F_0 \|_{ L^\infty([0,+\infty),  H^{3/2}(\S) )} + c L^{1+2\sigma},
\]
so that, in view of \eqref{sourceinfinito} and Remark \ref{rem-bastauno}, if $L<(2c)^{-2\sigma}$ and $\|\phi_0^\la\|_{H^2(\R^3)}$ is such that $ \| F_0 \|_{ L^\infty([0,+\infty),  H^{3/2}(\S) )}<L/2$, then $\| \LL g \|_{ L^\infty([0,+\infty),  H^{3/2}(\S) )}<L$, which proves the claim. 

Finally, it is left to discuss the contractivity with respect to a suitable metric. Consider the $L^\infty\big([0,+\infty),  L^2(\S) \big)$--one. Arguing as in the proof of Proposition \ref{carica}, this clearly makes $Y$ complete. Moreover, using \eqref{eq-banaleutile} and \eqref{eq-nonlin} and slightly adapting the proof of \eqref{eq-stimaL2}, we have that
\beq 
\| \LL (q_1) -\LL (q_2) \|_{ L^\infty([0,+\infty),  L^2(\S) ) } \leq c L^{2\sigma} \| q_1 -q_2 \|_{ L^\infty([0,+\infty),  L^2(\S) ) }
\eeq 
and therefore, with the previous choice of $L$, also contractivity is established.
\end{proof}


\subsection{Global well--posedness in the defocusing case: proof of Theorem \ref{main2} -- item (ii)}
\label{subsec-globaldef}

In view, again, of \eqref{eq-tmaxeq}, the latter strategy to prove global well--posedness is to adapt a classical \emph{blow--up alternative} argument to \eqref{eq-charge}.

\begin{lem}[Blow-up alternative]
 \label{lem-bualternative}
 Let $\beta\in\R$, $\sigma\geq1/2$, $\psi_0=\phi_0^\la-\G\nu(q_0)\in\D(\h)$ and $q$ the unique solution of \eqref{eq-charge} on $[0,T^*)$. Then,
 \[
  \text{either}\qquad T^*=+\infty\qquad\text{or}\qquad\limsup_{t\to T^*}\|q(t)\|_{H^{3/2}(\S)}=+\infty.
 \]
\end{lem}

\begin{proof}
 Let $T^*<+\infty$ and assume, by contradiction, that
 \begin{equation}
  \label{eq-ipass}
  \|q(t)\|_{H^{3/2}(\S)}\leq C,\qquad\forall t\in[0,T^*).
 \end{equation}
 In addition, set $T_\ep:=T^*-\ep$, for some $\ep$ to be fixed later, and define $\widetilde{q}(s):=q(T_\ep+s)$. By \eqref{eq-charge} and \eqref{eq-Lambda}, one finds that $\widetilde{q}$ satisfies
 \begin{equation}
  \label{eq-chargebis}
  \widetilde{q}(s,\x)+\imath\big(\Lambda\nu(\widetilde{q})\big)(s,\x)=\widetilde{F}_0(s,\x),
 \end{equation}
 where
 \begin{equation}
 \label{eq-newF}
  \widetilde{F}_0(s,\x):=F_0(T_\ep+s,\x)+A(s,\x),
 \end{equation}
 with $F_0$ given by \eqref{source} and 
 \begin{equation}
 \label{eq-newterm}
  A(s,\x):=-\imath\int_0^{T_\ep}\big(I_{T_\ep+s-\tau}\nu(q(\tau))\big)(\x)\dtau.
 \end{equation}
 Now, if one may choose $\ep$ such that \eqref{eq-chargebis} admits a unique solution in $C^0\big([0,\widetilde{T}],H^{3/2}(\S)\big)$ with $\widetilde{T}>\ep$, then, setting $q(t)=\widetilde{q}(t-T_\ep)$, one finds a solution of \eqref{eq-charge} beyond $T^*$, which contradicts \eqref{eq-tmaxeq}.
 
 Preliminarily, arguing as in the proof of Proposition \ref{carica}, one can see that $\widetilde{T}$ is nothing but a suitable contraction time for the map
 \[
  \widetilde{\LL}(\widetilde{q}):=-\imath\Lambda\nu(\widetilde{q})+\widetilde{F}_0
 \]
 in $C^0\big([0,\widetilde{T}],H^{3/2}(\S)\big)$. Moreover, a direct inspection of the proof of Proposition \ref{carica} shows that any positive time $T<2^{-6\sigma}\|F_0\|_{L^\infty([0,T],H^{3/2}(\S))}^{-6\sigma}$ is a suitable contraction time for the map defined by \eqref{eq-mappa}. As a consequence, any time $\widetilde{T}<2^{-6\sigma}\|\widetilde{F}_0\|_{L^\infty([0,\widetilde{T}],H^{3/2}(\S))}^{-6\sigma}$ is a suitable contraction time for $\widetilde{\LL}$. Let us estimate, then, $\|\widetilde{F}_0\|_{L^\infty([0,T],H^{3/2}(\S))}$. By \eqref{sourceinfinito}
 \[
  \|F_0(T_\ep+\cdot)\|_{L^\infty([0,T],H^{3/2}(\S))}\leq c_\la\big(\|\phi_0^\la\|_{H^2(\R^3)}+\|q_0\|_{H^{3/2}(\S)}^{2\sigma+1}\big),\qquad\forall T>0.
 \]
 Furthermore, using \eqref{stimasobo} with $\mu=3/2$ and $z=1$ and \eqref{schauder} one obtains that
 \[
  \|I_{T_\ep+s-\tau}\nu(q(\tau))\|_{H^{3/2}(\S)}\leq\f{c}{(T_\ep+s-\tau)^{2/3}}\|q(\tau)\|_{H^{3/2}(\S)}^{2\sigma+1}
 \]
 and thus, arguing as in the Proof of Proposition \ref{prop-cont} and using \eqref{eq-ipass},
 \[
  \|A\|_{L^\infty([0,T],H^{3/2}(\S))}\leq3cCT^*,\qquad\forall T>0.
 \]
 Therefore, by \eqref{eq-newF}, there exists a constant $C_{\la,\sigma,C,T^*,\phi_0^\la,q_0}>0$ such that
 \[
  \|\widetilde{F}_0\|_{L^\infty([0,T],H^{3/2}(\S))}\leq C_{\la,\sigma,C,T^*,\phi_0^\la,q_0},\qquad\forall T>0.
 \]
 Owing to this fact, any $\widetilde{T}<2^{-6\sigma}C_{\la,\sigma,C,T^*,\phi_0^\la,q_0}^{-6\sigma}$ is a suitable contraction time. Finally, since $C_{\la,\sigma,C,T^*,\phi_0^\la,q_0}^{-6\sigma}$ does not depend on $\ep$, one can chose (for instance) $\widetilde{T}=2^{-6\sigma-1}C_{\la,\sigma,C,T^*,\phi_0^\la,q_0}^{-6\sigma}$ and $\ep=2^{-6\sigma-2}C_{\la,\sigma,C,T^*,\phi_0^\la,q_0}^{-6\sigma}$ to obtain the aimed contradiction.
\end{proof}

Finally, we have all the ingredients to prove global well--posedness in the defocusing case for $\sigma< 4/5$ (we mention that the idea of the proof takes its cue from \cite{GV-84}).

\begin{proof}[Proof of Theorem \ref{main2} -- item (ii)]
We want to prove that, assuming $T^*<+\infty$, yields $\|q(t)\|_{H^{3/2}(\S)}\leq C$, for every $t\in[0,T^*)$, thus contradicting Lemma \ref{lem-bualternative}. Note that, whenever one assumes $T^*<+\infty$, $C$ may depend on $T^*$.

Suppose then, by contradiction, that $T^*<+\infty$. Combining \eqref{massa} and \eqref{energia} with $\beta>0$, we have
\beq
\label{costanti}
\| \psi(t) \|_{H^1(\R^3)} \leq c, \qquad \| q(t)\|_{L^{2\sigma +2}(\S)} \leq c,\qquad\forall t\in [0,T^*),
\eeq
so that, by classical Trace theorems and Sobolev embeddings,
\beq 
\label{costantiq}
\| q(t) \|_{H^{1/2}(\S)} \leq c, \qquad \| q(t) \|_{L^{4} (\S)} \leq c,\qquad\forall t\in [0,T^*).
\eeq
Now, in order to prove the contradiction, it is convenient to divide the proof in two steps.

\emph{Step (i): $\|q(t)\|_{L^\infty(\S)}\leq C_{T^*}$, for every $t\in[0,T^*)$.} Our strategy is to prove that there exists $\overline{r}>8/3$ such that $\|q(t)\|_{W^{3/4,\overline{r}}(\S)}\leq C_{T^*}$, for every $t\in[0,T^*)$, as this immediately implies the claim by Sobolev embeddings.

Let us start by estimating $[q(t)]_{\sn{W}^{3/4,r}(\S)}$, for a general $r\geq2$, using the definition given by \eqref{eq-sobp}. From \eqref{eq-charge}
\beq
\label{acqua}
\big\| (-\Delta_{\S})^{3/8} q (t) \big\|_{L^r (\S) } \leq \big\|  (-\Delta_{\S})^{3/8}  F_0 (t) \big\|_{L^r (\S) } + \big\| (-\Delta_{\S})^{3/8} \big(\L \nu(q)\big) (t) \big\|_{L^r (\S) } .
\eeq
Moreover, from \eqref{eq-diagLB} and \eqref{symbol}, one sees that $I_t$ commutes with $ (-\Delta_{\S})^{3/8}$, as they are multiplication operators with respect to the decomposition in spherical harmonics. As a consequence, exploiting \eqref{eq-Lambda} and Lemma \ref{lem-disp}, we have
\begin{equation}
\label{eq-stimaI}
\big\| (-\Delta_{\S})^{3/8} \big(\L \nu(q)\big) (t) \big\|_{L^r (\S) }  \leq c \int_0^t \frac{1}{|t-s|^{\de(p)} } \big\| (-\Delta_{\S})^{3/8} \nu\big(q(s)\big)\big\|_{L^p (\S) }\ds
\end{equation}
with $p$ the conjugate exponent of $r$ and $\de(p)$ given by \eqref{eq-deltap}. However, \eqref{eq-stimaI} is only formal unless $\de(p)<1$. Hence, it is actually valid for the sole $r\in\big[2,\f{10}{3}\big)$, with $p=\f{r}{r-1}\in\big(\f{10}{7},2 \big]$. Furthermore, combining the equivalence between \eqref{eq-sobp} and \eqref{eq-sobp2}, the fact that \eqref{eq-sobp} is valid also with $\R^2$ in place of $\S$ (with $-\Delta_{\S}$ replaced by the standard Laplacian) and \cite[Lemma A1]{K-95}, and arguing as in the proof of \eqref{schauder_true} (see Appendix \ref{sec-schauder}), one can prove that
\[
\big\| (-\Delta_{\S})^{3/8} \nu\big(q(s)\big)\big\|_{L^p (\S) }\leq c \big\| |q(t)|^{2\sigma} \big\|_{L^{p_1} (\S) }   \big\|  (-\Delta_{\S})^{3/8}q(t) \big\|_{L^{p_2} (\S) } , \quad \text{with}\:\: \frac{1}{p_1} + \frac{1}{p_2 } = \frac{1}{p}.
\]
Therefore, letting $p_2= r$, \eqref{eq-stimaI} reads 
\beq \label{fuochino}
\big\| (-\Delta_{\S})^{3/8} \big(\L \nu(q)\big) (t) \big\|_{L^r (\S) }  \leq c \int_0^t \frac{1}{|t-s|^{\de(p)} }\big\| |q(s)|^{2\sigma} \big\|_{L^{\f{r}{r-2}} (\S) }   \big\|  (-\Delta_{\S})^{3/8}q(s) \big\|_{L^{r} (\S) }\ds.
\eeq
Note that, since $r\in\big[2,\f{10}{3}\big)$, $\frac{r}{r-2}\in\big(\f{5}{2},+\infty\big]$, and thus, as $\sigma <4/5 $, there exists $\overline{r}\in \big(\f{8}{3},\f{10}{3}\big)$ such that
\begin{equation}
\label{eq-crucial}
\big\| |q(s)|^{2\sigma} \big\|_{L^{\f{\overline{r}}{\overline{r}-2}} (\S) } \leq c \| q(s) \|_{L^{4} (\S) }^{2\sigma}.
\end{equation}
Now, using \eqref{costantiq} and \eqref{sourceinfinito} and the facts that
\[
 \big\|  (-\Delta_{\S})^{3/8}  F_0\big\|_{C^0([0,T],H^{3/4}(\S) )}\leq \big\|  F_0\big\|_{C^0([0,T],H^{3/2}(\S) )},\qquad\forall T>0,
\]
and 
\[
 H^{3/4}(\S) \hookrightarrow L^8 (\S)  \hookrightarrow L^{\overline{r}} (\S),
\]
there results
\beq \label{fuoco}
\big\| (-\Delta_{\S})^{3/8} q (t) \big\|_{L^{\overline{r}} (\S) }  \leq c_1+ c_2 \int_0^t \frac{1}{|t-s|^{\de(p)} } \big\| (-\Delta_{\S})^{3/8} q (s) \big\|_{L^{\overline{r}} (\S) }\ds.
\eeq
Finally, from the Gr\"ownall--type estimate given by \cite[Lemma 2.5]{T-85} (and, more precisely, by \cite[Eq. (2.13)]{T-85}), \eqref{fuoco} implies that $\big\| (-\Delta_{\S})^{3/8} q (t) \big\|_{L^{\overline{r}} (\S) }\leq C_{T^*}$, for every $t\in[0,T^*)$. Since $\overline{r}\leq4$, the same can be proved for $\|q (t) \|_{L^{\overline{r}} (\S) }$ and thus the claim follows by Sobolev embeddings.

\emph{Step (ii): $\|q(t)\|_{H^{3/2}(\S)}\leq C_{T^*}$, for every $t\in[0,T^*)$.} Using again \eqref{eq-charge}, we obtain
\[
\|q (t) \|_{H^{3/2} (\S) } \leq \|F_0 (t) \big\|_{H^{3/2} (\S) } + \big\|\big(\L \nu(q)\big) (t) \big\|_{H^{3/2} (\S) }.
\]
Furthermore, combining \eqref{eq-stimautile}, \eqref{sourceinfinito}, \eqref{schauder_true} and the claim of Step (i), one obtains
\[
 \|q (t) \|_{H^{3/2} (\S) } \leq c_1+ c_2\, C_{\sigma, T^*} \int_0^t \frac{1}{|t-s|^{2/3} } \|q (s) \|_{H^{3/2} (\S) }\ds,
\]
so that, again by \cite[Lemma 2.5]{T-85}, we get the claim (possibly, redefining $C_{T^*}$).
\end{proof}

\begin{rem}
\label{rem-global}
In the previous proof, the constraint on the power $\sigma$ is a direct consequence of three things:
\begin{itemize}
 \item[(a)] the choice of a strategy based on a Gr\"onwall--type argument;
 \item[(b)] the use of Lemma \ref{lem-disp}, which thus requires $\de(p)<1$ when combined with (a);
 \item[(c)] the fact that the a--priori estimate on the charge provided by the constants of motion is only on the $L^4(\S)$--norm.
\end{itemize}
One may intuitively think to improve step (c) by using the estimate on the $L^{2\sigma+2}(\S)$--norm, which is another consequence of the energy conservation. However, if one replaces $L^4(\S)$ with $L^{2\sigma+2}(\S)$ in \eqref{eq-crucial}, then the inequality holds only for $\sigma<2/3$, which is an even worse constraint. Hence, the sole remaining possibility to remove the constraint at this level is to find another way to estimate $\big\| |q(s)|^{2\sigma} \big\|_{L^{\f{\overline{r}}{\overline{r}-2}} (\S) }$, but this seems out of reach at the moment.

On the other hand, one may think to improve step (b) proving a finer version of Lemma \ref{lem-disp}, that is a finer version of \eqref{stimaL2}. For instance, if one could replace the factor $t^{-2/3}$ with the factor $t^{-1/2}$, then one would enlarge the set of admissible exponents. Unfortunately, this is not possible due to the behavior as $t\to 0$, as shown by the following computation. First, by \cite[Eq. (10.19.8)]{NIST}, for $a>0$ fixed,
\[
J_\nu \big(\nu+ a \nu^{1/3} \big) \sim  \frac{2^{1/3}}{\nu^{1/3}} \, \mathrm{Ai}\big(-2^{1/3} a \big),\qquad\text{as}\quad\nu\to+\infty ,
\]
where $\mathrm{Ai}(x)$ is the Airy function. Hence, fixing $a$ such that $-2^{1/3} a $ is not a zero of $\mathrm{Ai}(x)$, there exists $\nu_0>0$ such that
\[
 \nu^{1/3} \big|J_\nu \big(\nu+ a \nu^{1/3} \big)\big| \geq c>0,\qquad\forall\nu>\nu_0.      
\]
Let us set, then,
\[
\frac{1}{2t_n}:= n+1/2 + a(n+1/2)^{1/3},\qquad\forall n\in\N\setminus\{0\}.
\]
By \eqref{symbol}, whenever $n$ is large enough,
\begin{multline*}
\sup_{\ell \in \N } |\rho(t_n, \ell)| =\frac{1}{2|t_n|} \sup_{\ell \in \N } | J_{\ell+1/2} (1/ 2t_n)|  \\
\geq   \frac{1}{2 |t_n|}\big| J_{n+1/2} \big( n+1/2 + a(n+1/2)^{1/3}\big)\big|\geq\frac{c}{2(n+1/2)^{1/3}|t_n|}
\end{multline*}
and thus, up to lower order terms, $\sup_{\ell \in \N } |\rho(t_n, \ell)|\geq  \frac{\widetilde{c}}{|t_n|^{2/3} }$, which shows that the order $2/3$ in \eqref{stimaL2} is sharp as $t\to0$.

As a consequence, in order to remove (or enhance) the constraint on $\sigma$, the unique possibility is to work on step (a), which means to change the overall strategy of the proof. However, at the moment it is not clear how to recover a global well--posedness result in $\D(\h)$, without small data assumptions, exploiting different methods.
\end{rem}


\appendix


\section{Schauder estimates on $\S$: proof of (\ref{schauder_true})}
\label{sec-schauder}

Here we prove \eqref{schauder_true}. First, we recall that in the euclidean case, when $\S$ is replaced by $\R^2$, the inequality is well known to be true; namely,
\beq
\label{schauder_R}
\| \nu (g) \|_{ H^{\mu}(\R^2) } \leq c_\mu\,\| g \|^{2\sigma}_{ L^\infty(\R^2) } \| g \|_{ H^{\mu}(\R^2) },\qquad\forall g\in H^\mu(\R^2),
\eeq
whenever $\sigma\geq\f{[\mu]}{2}$ (this can be easily derived, for instance, by \cite{K-95} or by \cite[Lemma A.9]{T-06}). As a consequence, we aim at using the construction of the Sobolev spaces on $\S$ introduced at the end of Section \ref{sec-sobs} to transfer the inequality from $\R^2$ to $\S$.

Fix, then, $g\in H^\mu(\S)$, with $\mu>1$. First, we note that
\begin{equation}
\label{eq-rel1}
 \pi_j\big[\chi_j\nu(g)\big]=\pi_j[\chi_j]\nu\big(\pi_j[u]\big),\qquad j=1,2.
\end{equation}
Consider, now, a new partition of the unity $\{\eta_1,\eta_2\}$, which has all the properties of $\{\chi_1,\chi_2\}$ and, furthermore satisfies
\begin{equation*}
 \eta_1\equiv1,\qquad\text{on}\quad{\rm supp}\{\chi_1\}.
\end{equation*}
Since this new partition yields an equivalent norm for $H^\mu(\S)$, there exist $c_1,\,c_2>0$ such that
\begin{equation}
 \label{eq-normheq}
 c_1\,\big\|\pi_1[\chi_1 g]\big\|_{H^\mu(\R^2)}\leq\big\|\pi_1[\eta_1 g]\big\|_{H^\mu(\R^2)}\leq c_2\,\big\|\pi_1[\chi_1 g]\big\|_{H^\mu(\R^2)}.
\end{equation}
In addition, one can check that
\[
 \pi_1[\eta_1]\equiv1\qquad\text{on}\quad {\rm supp}\left\{\pi_1\big[\chi_1\nu(g)\big]\right\}
\]
and thus, combining with \eqref{eq-rel1}, one obtains
\[
 \pi_1\big[\chi_1\nu(g)\big]=\pi_1[\chi_1]\nu\big(\pi_1[\eta_1g]\big).
\]
As a consequence, using \cite[Lemma A.8]{T-06}, there results
\begin{multline*}
 \big\|\pi_1\big[\chi_1\nu(g)\big]\big\|_{H^\mu(\R^2)}\\[.2cm]
 \leq c\,\left(\big\|\pi_1[\chi_1]\big\|_{L^\infty(\R^2)}\big\|\nu\big(\pi_1[\eta_1g]\big)\big\|_{H^\mu(\R^2)}+\big\|\pi_1[\chi_1]\big\|_{H^\mu(\R^2)}\big\|\nu\big(\pi_1[\eta_1g]\big)\big\|_{L^\infty(\R^2)}\right),
\end{multline*}
so that, from \eqref{infi} and \eqref{schauder_R},
\[
 \big\|\pi_1\big[\chi_1\nu(g)\big]\big\|_{H^\mu(\R^2)}\leq c\,\big\|\pi_1[\eta_1g]\big\|_{L^\infty(\R^2)}^{2\sigma}\big\|\pi_1[\eta_1g]\big\|_{H^\mu(\R^2)}.
\]
Finally, recalling \eqref{eq-normheq} and \eqref{eq-normlab} and noting that
\[
 \gamma_1\,\|u\|_{L^\infty(\S)}\leq\big\|\pi_1[\eta_1g]\big\|_{L^\infty(\R^2)}+\big\|\pi_2[\eta_2g]\big\|_{L^\infty(\R^2)}\leq \gamma_2\,\|u\|_{L^\infty(\S)}
\]
for some suitable $\gamma_1,\,\gamma_2>0$, one finds that
\begin{equation*}
 \big\|\pi_1\big[\chi_1\nu(g)\big]\big\|_{H^\mu(\R^2)}\leq c\,\|g\|_{L^\infty(\S)}^{2\sigma}\|g\|_{H^\mu(\S)}.
\end{equation*}
Since one can make the same construction for $\pi_2\big[\chi_2\nu(g)\big]$, inequality \eqref{schauder_true} follows just recalling again \eqref{eq-normheq}.


\section{Proof of (\ref{eq-integralbessel})}
\label{app-formula}

Here we prove \eqref{eq-integralbessel}. Note that it is actually present in \cite[Section 13.31]{W-95}, but its proof holds for another range of parameters. Hence, it is necessary to prove it again for our range of parameters.

Fix, then, $\ell\in\N$ and $s>0$. Since $J_{\ell +1/2} (z)$ is holomorphic outside the origin (where it displays a ramification point), there results
\[
\oint_{\Gamma_{R,\ve}} f(z)\dz =0,
\]
where
\[
 f(z): =z  J_{\ell+1/2}^2 (z)  \exp^{\i z^2 s   }\qquad\text{and}\qquad\Gamma_{R,\ve}= \gamma_1 \cup \gamma_2 \cup \gamma_3 \cup \gamma_4
\]
is the curve depicted by Figure \ref{sonobravo}. Observe that $\gamma_2 $ is an arc of angle $\pi/4$ of the circle of radius $R\gg s$, whereas $\gamma_4 $ is an arc of angle $\pi/4$ of the circle of radius $\ve\ll s$ (the definitions of $\gamma_1$ and $\gamma_3$ are immediate). Clearly

\begin{figure}[h]
\begin{tikzpicture}
\tikzset{>=latex}
\draw[->] (-1,0) to (6,0) ;
\draw[->] (0,-1) to (0,5) ;
\draw[thick, decoration={
             markings,
             mark=at position 0.5 with \arrow[thick]{>[scale=4.0, length=1.5, width=1.5]}}
       ,postaction=decorate] (0.5,0) to (5.5,0) ;
\draw[thick, decoration={
             markings,
             mark=at position 0.5 with \arrow[thick]{>[scale=4.0, length=1.5, width=1.5]}}
       ,postaction=decorate] (45:5.5) to  (45:0.5) ;
\draw[thick, decoration={
             markings,
             mark=at position 0.8 with \arrow[thick]{>[scale=4.0, length=1.5, width=1.5]}}
       ,postaction=decorate]  (45:0.5) arc [radius=0.5, start angle=45, end angle= 0];
\draw[thick, decoration={
             markings,
             mark=at position 0.5 with \arrow[thick]{>[scale=4.0, length=1.5, width=1.5]}}
       ,postaction=decorate]   (5.5,0) arc [radius=5.5, start angle=0, end angle= 45];
\node[below] at (4,0) {\scalebox{1}{$\gamma_1$} };
\node[right] at (4.7,3.1) {\scalebox{1}{$\gamma_2$} };
\node[above] at (1.5,2) {\scalebox{1}{$\gamma_3$} };
\node[right] at (0.5,.35) {\scalebox{1}{$\gamma_4$} };
\end{tikzpicture}
\caption{The path $\Gamma_{R,\ve}.$} \label{sonobravo}
\end{figure}
\[
\lim_{R\to +\infty} \lim_{\ve\to 0} \oint_{\Gamma_{R,\ve}} f(z)\dz =0,
\]
and
\[
 \lim_{\ve\to 0} \int_{\gamma_4} f(z)\dz =0,
\]
so that, if one can prove that
\beq \label{pallosa}
\lim_{R\to +\infty} \int_{\gamma_2} f(z) \dz= 0,
\eeq
then the proof is complete. Indeed, owing to this fact, one has
\begin{multline*}
 \int_0^{+\infty} k J_{\ell+1/2}^2 (k)  \exp^{\i k^2 s } \dkm= \lim_{R\to +\infty} \lim_{\ve\to 0}\int_{\gamma_1} f(z)\dz\\
 =-\lim_{R\to +\infty} \lim_{\ve\to 0}\int_{\gamma_3} f(z)\dz=\lim_{R\to +\infty} \lim_{\ve\to 0}\int_{-\gamma_3} f(z)\dz
\end{multline*}
and, since
\begin{multline*}
\lim_{R\to +\infty} \lim_{\ve\to 0}\int_{-\gamma_3} f(z)\dz= \imath\lim_{R\to +\infty} \lim_{\ve\to 0}  \int_\ve^R t \exp^{-s t^2} J_{\ell+1/2}^2 \big(t \exp^{\i \frac{\pi}{4}}\big)\dt \\
= \imath\int_0^{+\infty} t \exp^{-s t^2} J_{\ell+1/2}^2 \big(t \exp^{\i \frac{\pi}{4}}\big)\dt =\frac{\imath}{2s} \exp^{-\frac{\i}{2s}} I_{\ell+1/2} \lf(  \frac{\i}{2s} \ri)
\end{multline*}
(where the last equality is established in \cite[Pag. 395]{W-95}), there results
\[
 \int_0^{+\infty} k J_{\ell+1/2}^2 (k)  \exp^{\i k^2 s } \dkm =\frac{\imath}{2s} \exp^{-\frac{\i}{2s}} I_{\ell+1/2} \lf(  \frac{\i}{2s} \ri),
\]
with $I_\nu$ the modified Bessel function of order $\nu$ (see, e.g., \cite[Eq. (10.25.2)]{NIST}). Finally, since, by \cite[Eq. (10.27.6)]{NIST}, $I_\nu (z) = \exp^{\i \nu \frac{\pi}{2}} J_\nu (z  \exp^{-\i \frac{\pi}{2}} )$, whenever $-\pi \leq \arg z\leq \pi/2$, one obtains \eqref{eq-integralbessel}. 

As a consequence, it is left to prove \eqref{pallosa}, or equivalently
\begin{equation}
\label{eq-pallosa2}
\lim_{R\to +\infty} \underbrace{R^2 \int_0^{\f{\pi}{4}} \exp^{\i (2\theta+s R^2 \exp^{\i 2\theta})} J^2_{\ell+1/2} \big(R \exp^{\i \theta}\big)\dteta}_{=:A(R)}=0.
\end{equation}
Note that all the constants below may depend on $s$ and $\ell$. First, we see that, since $\sin\eta\geq\f{2\eta}{\pi}$ for every $\eta\in\big[0,\f{\pi}{2}]$,
\begin{align*}
|A(R)|
&\leq R^2 \int_0^{\f{\pi}{4} } \exp^{-s R^2 \sin (2\theta)} \big| J_{\ell+1/2} \big(R \exp^{\i \theta}\big) \big|^2 \dteta \leq R^2 \int_0^{\f{\pi}{4} } \exp^{-s R^2 \frac{4}{\pi} \theta} \big| J_{\ell+1/2} \big(R \exp^{\i \theta}\big) \big|^2 \dteta\\[.2cm]
& \leq \int_0^{\f{\pi}{4}R^2 } \exp^{-\frac{4s}{\pi}x } \big| J_{\ell+1/2} \big(R \exp^{\i \f{x}{R^2} }\big) \big|^2 \dxm.
\end{align*}
Moreover, by the Schl\"afli's formula (see, e.g., \cite[Eq. (10.9.6)]{NIST}), whenever $|\arg z|  <\pi/2$,
\[
J_{\ell+1/2} (z) =\frac{1}{\pi} \int_0^\pi \cos \lf( z \sin \theta - (\ell+1/2) \theta\ri) \dteta
-\frac{\sin \big( (\ell+1/2) \pi\big)}{\pi} \int_0^{+\infty}\exp^{-z \sinh t-(\ell+1/2)t} \dt.
\]
Then,
\begin{align*}
 J_{\ell+1/2} \big(R \exp^{\i \f{x}{R^2} }\big) = &f_1 (x)+f_2(x) +f_3(x)\\[.2cm]
:=& \frac{1}{2\pi} \int_0^\pi \exp^{\i R\lf( \cos \lf(\f{x}{R^2}\ri) + \i \sin \lf(\f{x}{R^2}\ri) \ri) \sin \theta } \exp^{-\i \theta(\ell+1/2)} \dteta+ \\[.2cm]
&+ \frac{1}{2\pi} \int_0^\pi \exp^{-\i R\lf( \cos \lf(\f{x}{R^2}\ri) + \i \sin \lf(\f{x}{R^2}\ri) \ri) \sin \theta } \exp^{\i \theta(\ell+1/2)} \dteta+  \\[.2cm]
& - \frac{\sin \big( (\ell+1/2) \pi\big)}{\pi} \int_0^{+\infty}\exp^{-R \exp^{\i \f{x}{R^2} }\sinh t-(\ell+1/2)t}  \dt,
\end{align*}
so that
\[
 |A(R)|\leq \sum_{j=1}^3\underbrace{\int_0^{\f{\pi}{4}R^2 } \exp^{-\frac{4s}{\pi}x } |f_j(x)|^2 \dxm}_{=:A_j(R)}.
\]
As a consequence, in order to prove \eqref{eq-pallosa2}, it suffices to show that $A_j(R)\to0$, as $R\to+\infty$, for $j=1,2,3$. Let us consider, first, the cases $j=1,2$. An easy change of variables yields
\[
f_j( x) =  \frac{\imath^{(-1)^j(\ell+1/2)}}{2\pi} \int_{-\pi /2}^{\pi /2} \exp^{\i(-1)^{j-1} R\exp^{\i \f{x}{R^2}}\cos \alpha }\exp^{\i (-1)^j(\ell+1/2)\alpha}\da.
\]
Now, let $\phi\in C_0^\infty (\R)$, with $\mathrm{supp}\{\phi\}\subset[a,b]\subset(-\pi /2, \pi/2)$, $\phi\geq0$ and $\phi\equiv1$ on $[-\pi /4, \pi/4]$, and define
\[
 \widetilde{f}_j( x) :=  \frac{\imath^{(-1)^j(\ell+1/2)}}{2\pi} \int_{-\pi /2}^{\pi /2} \phi(\alpha)\exp^{\i(-1)^{j-1} R\exp^{\i \f{x}{R^2}}\cos \alpha }\exp^{\i (-1)^j(\ell+1/2)\alpha}\da.
\]
It is straightforward that 
\[
 A_j(R)\leq c \underbrace{\int_0^{\f{\pi}{4}R^2 } \exp^{-\frac{4s}{\pi}x } |f_j(x)-\widetilde{f}_j( x)|^2 \dxm}_{=:A_{j,1}(R)}+c\underbrace{\int_0^{\f{\pi}{4}R^2 } \exp^{-\frac{4s}{\pi}x } |\widetilde{f}_j( x)|^2 \dxm}_{=:A_{j,2}(R)}
\]
and, thus, we can prove that $A_{j,1}(R)$ and $A_{j,2}(R)$ are both vanishing as $R\to+\infty$ to get the result. 

In order to study the former term, we see that by construction
\[
 f_j(x)-\widetilde{f}_j( x)=\frac{\imath^{(-1)^j(\ell+1/2)}}{2\pi}\sum_{n=1}^2\int_{a_n}^{b_n}\big(1-\phi(\alpha)\big)\exp^{\i(-1)^{j-1} R\exp^{\i \f{x}{R^2}}\cos \alpha }\exp^{\i (-1)^j(\ell+1/2)\alpha}\da,
\]
with $a_1=-b_2=-\pi/2$ and $b_1=-a_2=-\pi/4$, so that, integrating by parts and using the properties of $\phi$,
\begin{multline*}
 f_j(x)-\widetilde{f}_j( x)=\f{\i^{(-1)^j(\ell+1/2)}}{4\pi R e^{\i\f{x}{R^2}}}\bigg\{2\cos\big((-1)^j(\ell+1/2)\pi/2\big)+\\[.2cm]
 -\sum_{n=1,2}\int_{a_n}^{b_n}\bigg(\f{(1-\phi(\alpha))\exp^{\i(-1)^j(\ell+1/2)\alpha}}{\sin\alpha}\bigg)'\exp^{\i(-1)^{j-1}R\exp^{\i\f{x}{R^2}}\cos\alpha}\da\bigg\}.
\end{multline*}
Hence, using the range of $\alpha$ and $x$ and (again) the properties of $\phi$, one obtains
\[
 |f_j(x)-\widetilde{f}_j( x)|\leq c\,\f{\big(1+\exp^{\f{x}{R}}\big)}{R}
\]
and thus
\[
 A_{j,1}(R)\leq\f{c}{R^2}\bigg(1+\f{1}{\f{4s}{\pi}-\f{2}{R}}\big(1-\exp^{-sR^2+2R}\big)\bigg)\longrightarrow0,\qquad\text{as}\quad R\to+\infty
\]
(note that all the constants here and below may also depend on $\phi$, which is fixed).

Concerning the latter term, we first see that, setting $y= \sin (\alpha /2)$, there results
\begin{multline*}
 \widetilde{f}_j( x)=\f{\i^{(-1)^j(\ell+1/2)}}{\pi}\exp^{\i(-1)^{j-1}R\cos\lf(\f{x}{R^2}\ri)}\,\cdot\\[.2cm]
 \cdot\,\int_{-\sqrt{2}/2}^{\sqrt{2}/2}\f{\phi(2\arcsin y)}{\sqrt{1-y^2}}\exp^{2\imath(-1)^jR\cos\lf(\f{x}{R^2}\ri)y^2}\exp^{(-1)^{j-1}R\sin\lf(\f{x}{R^2}\ri)(1-y^2)}\exp^{\imath(-1)^j(2\ell+1)\arcsin y}\dy.
\end{multline*}
Moreover, using again the construction of $\phi$, we have that
\[
 \widetilde{f}_j( x)=\f{\i^{(-1)^j(\ell+1/2)}}{\pi}\exp^{\i(-1)^{j-1}R\cos\lf(\f{x}{R^2}\ri)}\int_{\widetilde{a}}^{\widetilde{b}}\exp^{\i h(x,r)(-1)^jy^2}g(x,R,y)\dy
\]
with $\widetilde{a}:= \sin (a /2)$, $\widetilde{b}:= \sin (b /2)$,
\[
h(x,r): = 2R  \cos \lf(\f{x}{R^2}\ri),\qquad 
g(y): = \f{\phi(2\arcsin y)}{\sqrt{1-y^2}}\exp^{(-1)^{j-1}R\sin\lf(\f{x}{R^2}\ri)(1-y^2)}\exp^{\imath(-1)^j(2\ell+1)\arcsin y}.
\]
Hence, exploiting the Van der Corput lemma (see, e.g., \cite[Corollary at pag. 334]{S-93} or \cite[Corollary 1.1]{LP-15}) and the fact that $x\in\big[0,\f{\pi}{4}R^2\big]$, we get that
\[
\big|\widetilde{f}_j( x)\big|^2 \leq \f{c}{R} \|g(x,R,\cdot)\|_{H^1(\widetilde{a},\widetilde{b})}^2.
\]
Furthermore, since easy computations yield
\[
 \|g(x,R,\cdot)\|_{H^1(\widetilde{a},\widetilde{b})}^2\leq c \exp^{\f{2x}{R}}\bigg(1+\f{x^2}{R^2}\bigg),
\]
there results
\begin{align*}
 A_{j,2}(R)\leq & \f{c}{R}\int_0^{\f{\pi}{4}R^2} \exp^{-\f{4s}{\pi}x}\exp^{\f{2x}{R}}\bigg(1+\f{x^2}{R^2}\bigg)\dxm\\[.2cm]
           \leq & \f{c}{R}+\f{c}{R}\int_0^{\f{\pi}{4}R^2}\exp^{-\f{4s}{\pi}x}\exp^{\f{2x}{R}}\f{x^2}{R^2}\dxm\\[.2cm]
              = & \f{c}{R}+\f{c}{R}\lf(\f{R}{\lf(2-\f{4sR}{\pi}\ri)^2}\int_0^{\f{\pi}{4}R\lf(2-\f{4sR}{\pi}\ri)}\exp^\omega\omega^2\,d\omega\ri)\\[.2cm]
              = & \f{c}{R}+\f{c}{R}\lf(\f{R}{\lf(2-\f{4sR}{\pi}\ri)^2}\bigg[\exp^\omega\big(\omega^2-2\omega+2\big)\bigg|_0^{\f{\pi}{4}R\lf(2-\f{4sR}{\pi}\ri)}\ri)\longrightarrow0,\qquad\text{as}\quad R\to+\infty.
\end{align*}

It is, then, left to prove the vanishing of $A_3(R)$. However, as one immediately sees that
\[
 f_3(x)=- \frac{\sin \big( (\ell+1/2) \pi\big)}{\pi} \int_0^{+\infty}\exp^{-\i R \sin\lf(\f{x}{R^2}\ri)\sinh t}\exp^{-R \cos\lf(\f{x}{R^2}\ri)\sinh t}\exp^{-(\ell+1/2)t} \dt,
\]
elementary estimates and a change of variables show that
\begin{align*}
 |f_3(x)|\leq & c\int_0^{+\infty}\exp^{-R \cos\lf(\f{x}{R^2}\ri)\sinh t}\exp^{-(\ell+1/2)t} \dt \leq c\int_0^{+\infty}\exp^{-R \cos\lf(\f{x}{R^2}\ri)\sinh t}\dt\\[.2cm]
         \leq & c\int_0^{+\infty}\exp^{-c_1R \sinh t}\dt=c\int_0^{+\infty}\exp^{-c_1Ry}\bigg(1+\f{y}{y+\sqrt{y^2+1}}\bigg)\dy\\[.2cm]
         \leq & c\int_0^{+\infty}\exp^{-c_1Ry}\dy\leq\f{c}{R},
\end{align*}
which implies $A_3(R)\to0$, as $R\to+\infty$.

\end{document}